\RequirePackage{fix-cm}
\documentclass[smallextended]{svjour3}       
\smartqed  
\usepackage{graphicx}
\usepackage{mathptmx}      
\usepackage{natbib,amsmath,amsfonts}
\usepackage{enumitem}
\RequirePackage[colorlinks,citecolor=blue,urlcolor=blue]{hyperref}
\newcommand{\N}{{\mathbb N}}
\newcommand{\R}{{\mathbb R}}
\def\P{\mathbb P}
\newcommand{\E}{{\mathbb E}}
\newcommand{\what}{\widehat}

\newcommand{\cov}{{\rm {Cov}}}

\numberwithin{equation}{section}
\numberwithin{theorem}{section}
\numberwithin{lemma}{section}
\numberwithin{proposition}{section}
\numberwithin{corollary}{section}
\numberwithin{example}{section}
\numberwithin{remark}{section}
\numberwithin{definition}{section}
\begin{document}

\title{On the rate of concentration of maxima in Gaussian arrays}



\author{Rafail Kartsioukas
\and
        Zheng Gao
\and 	
	    Stilian Stoev 
}


\institute{R. Kartsioukas \at
			 University of Michigan, Department of Statistics, Ann Arbor, Michigan, USA\\ 
			 \email{rkarts@umich.edu} 
           \and
           Z. Gao \at
              University of Michigan, Department of Statistics, Ann Arbor, Michigan, USA\\ 
			 \email{gaozheng@umich.edu} 
           \and
           S. Stoev \at
              University of Michigan, Department of Statistics, Ann Arbor, Michigan, USA\\ 
			 \email{sstoev@umich.edu} 
}

\date{Received: date / Accepted: date}

\maketitle
\begin{abstract}
Recently in \cite{gao2018fundamental} it was established that the concentration of maxima phenomenon
is the key to solving the exact sparse support recovery problem in high dimensions.   This phenomenon,
known also as relative stability, has been little studied in the context of dependence.  Here, we 
obtain bounds on the rate of concentration of maxima in Gaussian triangular arrays.  These results are used to establish sufficient conditions for the uniform relative stability of functions of Gaussian arrays, leading to new models that exhibit phase transitions in the exact support recovery problem.
Finally, the optimal rate of concentration for Gaussian arrays is studied under more general assumptions than the ones implied by the classic condition of \cite{berman1964limit}.

\end{abstract}

\section{Introduction} Let $Z_i,\ i=1,2,\dots$ be independent and identically distributed (iid) standard Normal random variables. It is well known
that their maxima under affine normalization converge to the Gumbel extreme value distribution.  If, however, one chooses to standardize the maxima 
by only dividing by a sequence of positive numbers, then the only possible limits are constants.  Specifically, for all $a_p\sim \sqrt{2\log(p)}$, we have
\begin{equation}\label{e:rel-stab}
\frac{1}{a_p} \max_{i\in [p]} Z_i \stackrel{{\mathbb P}}{\longrightarrow} 1,\quad \mbox{as }p\to\infty,
\end{equation}
where $[p] :=\{1,\cdots,p\}$ and in fact the convergence is valid almost surely. 
This property, known as {\em relative stability}, dates back to the seminal 
work of \citet{gnedenko1943distribution} who has characterized it in terms of rapid variation of the law of the $Z_i$'s
(see Section \ref{sec:gao-review} below, as well as \cite{barndorff1963limit,resnick1973almost,kinoshita1991convergence}).

In contrast, if the $Z_i$'s are iid and heavy-tailed, i.e., ${\mathbb P}[Z_i>x] \propto x^{-\alpha}$, for some $\alpha>0$, with $a_p \propto p^{1/\alpha}$, 
we have
\begin{equation}\label{e:disp-max}
\frac{1}{a_p} \max_{i\in [p]} Z_i \stackrel{d}{\longrightarrow} \xi, 
\end{equation}
where $\xi$ is a random variable with the $\alpha$-Fr\'echet distribution.

Comparing \eqref{e:rel-stab} and \eqref{e:disp-max}, we see that the maxima have fundamentally different asymptotic 
behavior relative to rescaling with constant sequences.  In the light-tailed regime, they {\em concentrate} around a constant 
in the sense of \eqref{e:rel-stab}, whereas in the heavy-tailed regime they {\em disperse} according to a probability distribution viz 
\eqref{e:disp-max}.

Although this concentration of maxima phenomenon may be well-known under independence, we found that it is virtually 
unexplored under {\em dependence}.  In this paper, we will focus on Gaussian sequences, and in fact, more generally,
Gaussian triangular arrays ${\cal E} = \{\epsilon_p(i),\ i\in [p],\ p\in\N\}$, where the $\epsilon_p(i)$'s are marginally standard 
Normal but possibly  dependent.  Let $u_p$ be the $(1-1/p)$-th quantile of the standard Normal distribution, i.e., 
$p\overline \Phi(u_p):=p\left(1-\Phi(u_p)\right)=1$. We say that the array ${\cal E}$ is uniformly relatively stable (URS), if
\begin{equation}\label{eq: URS S_p}
\frac{1}{u_{|S_p|}} \max_{i\in S_p} \epsilon_p(i) \stackrel{{\mathbb P}}{\longrightarrow }1,\quad \mbox{as } |S_p|\to\infty,
\end{equation}
for every choice of growing subsets $S_p\subset \{1,\cdots,p\}$. Note that $u_p\sim\sqrt{2\log(p)}$ (see \textit{e.g.} Lemma \ref{l:ap-bp-constants}).  Certainly, the relative stability property shows that
all iid Gaussian arrays are trivially URS.  The notion of uniform relative stability, however, is far from automatic or trivial under dependence. In the recent work of \cite{gao2018fundamental}, it was found that URS is the key to establishing the fundamental limits in sparse-signal support estimation in high-dimensions.  Specifically, under URS, a phase-transition phenomenon was shown to take place in the support recovery problem.
For more details, see Section \ref{sec:gao-phase-transition} below.

Theorem 3.1 in \cite{gao2018fundamental} gives a surprisingly simple necessary and sufficient condition for a 
Gaussian array ${\cal E}$ to be URS. As an illustration, in the special case where $\epsilon_p(i) \equiv Z_i,\ i\in \N$
form a stationary Gaussian time series, the array ${\cal E}$ is URS if and only if the auto-covariance vanishes, i.e.,
\begin{equation}\label{eq: cov to 0}
{\rm Cov}(Z_k,Z_0)\longrightarrow 0,\quad \mbox{as }k\to\infty.    
\end{equation}

That is, \eqref{e:rel-stab} holds (with $a_p\sim \sqrt{2\log(p)}$), for {\em any} stationary Gaussian time series
$Z = \{ Z_i\}$ with vanishing auto-covariance, no matter the rate of decay. The ``if'' part of \eqref{eq: cov to 0} appeared in Theorem 4.1 in \cite{berman1964limit}. 

Condition \eqref{eq: cov to 0} should be contrasted with the classic
Berman condition, $${\rm Cov}(Z_k,Z_0) = o\left(\frac{1}{\log(k)}\right),\quad \mbox{as } k\to\infty,$$ which entails {\em distributional} convergence under affine 
normalization.  Here, our focus is not on distributional limits but on merely the concentration of maxima under rescaling,
which can take place under much more severe dependence.  In fact, unlike Berman, here we are not limited to the time-series setting.  For
a complete statement of the characterization of URS, see Section \ref{sec:gao-review}, below.

While \cite{gao2018fundamental} characterized the conditions under which the convergence \eqref{eq: URS S_p} takes place, the rate of this convergence remained an open question. In this paper, our goal is to establish bounds on the {\em rate of concentration} for maxima of Gaussian arrays.  
Specifically, we establish results of the type
\begin{equation}\label{e:rate-delta_p}
{\mathbb P} \left [ \left |\frac{1}{u_p} \max_{i\in [p]} \epsilon_p(i) - 1\right | > \delta_p \right] \longrightarrow 0,
\end{equation}
where $\delta_p \to 0$ decays at a certain rate.  The rate of the sequence $\delta_p$ is quantified explicitly
in terms of the covariance structure of the array.  More precisely, the packing numbers $N(\tau)$ associated with the
UDD condition introduced in \cite{gao2018fundamental} will play a key role.  These packing numbers arise from a 
Sudakov-Fernique type construction, which appear to be close to optimal, although at this point we do not know if 
the so obtained bounds on the rates can be improved (cf Conjecture \ref{con: optimal rate}, below).  After completing this work,
we became aware of the important results of \cite{tanguy2015some},
which are closely related to ours in the special case of stationary
time series. Our approach, however, is technically different and
yields explicit rates for the general case of Gaussian triangular 
arrays. For more details, see Remark \ref{re: Tanguy}, below.

\medskip
Our general results are illustrated with several models, where explicit bounds on the rates of concentration are derived.
In Section \ref{sec:optimal}, we study the {\em optimal rate} of concentration and show that under rather broad dependence conditions
(including the iid setting), \eqref{e:rate-delta_p} holds if and only if $\delta_p \gg 1/\log(p)$.  Somewhat curiously, the
constant $u_p$ matters and the popular choice of $u_p:=\sqrt{2\log(p)}$ leads to the slower rates of 
$\log(\log(p))/\log(p)$.

Our bounds on the rate of concentration find important application in the study of uniform relative stability for functions of Gaussian arrays. Specifically, let $\eta_p(i)=f(\epsilon_p(i)),$ where $\mathcal{E}=\left\{\epsilon_p(i),\ i\in[p],\ p\in\mathbb{N}\right\}$ is a Gaussian triangular array and $f$ is a given deterministic function. In Section \ref{sec: functions Gaussian}, using our results on the rate of concentration for the array $\mathcal{E}$, we establish conditions which imply the uniform relative stability of the array $\mathcal{H}=\left\{\eta_p(i),\ i\in[p],\ p\in\mathbb{N}\right\}.$ Consequently, we establish that many dependent log-normal and $\chi^2$-arrays are URS, and hence obey the phase-transition result of \cite{gao2018fundamental}. 

\medskip
{\em The paper is structured as follows.} In Section \ref{sec:review}, we review the statistical inference problem motivating
the study of the concentration of maxima phenomenon.  Recalled is the notion of uniform decreasing
dependence involved in the characterization of uniform relative stability for Gaussian arrays.  A brief discussion on the optimal rate of concentration is given in Section \ref{sec:optimal}. Section \ref{sec:main} contains
the statement of the main result as well as some examples and applications. Section \ref{sec:proofs} contains proofs and  technical results, which
may be of independent interest.
\section{Concentration of maxima and high-dimensional inference} \label{sec:review}

In this section, we start with the statistical inference problem that motivated us to study the concentration of maxima 
phenomenon.  Readers who are convinced that this is a phenomenon of independent interest can skip to Section 
\ref{sec:gao-review}, where concrete definitions and notions are reviewed.

\subsection{Fundamental limits of support recovery in high dimensions}\label{sec:gao-phase-transition}

Our main motivation to study the relative stability or {\em concentration of maxima} under dependence is the fundamental role it plays in recent developments on high-dimensional statistical inference, which we briefly review next.  
Consider the classic {\em signal plus noise} model
$$
x_p(i) = \mu_p(i) + \epsilon_p(i),\quad i\in [p],
$$
where $\mu_p = (\mu_p(i)) \in \R^p$ is an unknown high-dimensional `signal' observed with additive noise.  The noise is modeled with a triangular array
${\cal E} = \{\epsilon_p(i),\ i\in[p],\ p\in\N\}$, where for concreteness, all $\epsilon_p(i)$'s are standardized to have the same marginal distribution 
$F$. However, this noise can have arbitrary dependence structure, in principle.  

One popular and important high-dimensional inference context, is the one where the dimension $p$ grows to infinity and the signal is sparse.  
Namely, the signal support set $S_p := \{ i\in [p]\, :\, \mu_p(i) \not=0\}$
is of smaller order than its dimension:
$$
|S_p| \sim p^{1-\beta},\mbox{ for some } \beta \in (0,1).
$$
The parameter $\beta$ controls the degree of sparsity; if $\beta$ is larger, the signal is more sparse, i.e., has fewer non-zero components.  In this context, many 
natural questions arise such as the detection of the presence of non-zero signal or the estimation of its support set (see, e.g., \cite{ingster1998minimax, donoho2004higher, ji2012ups,arias2017distribution}). Here, as in \cite{gao2018fundamental}, we focus on the fundamental {\em support recovery} problem. Particularly, under what conditions on 
the signal magnitude we can have {\em exact support recovery} in the sense that
$$
 {\mathbb P}[\what S_p = S_p] \longrightarrow 1,\quad \mbox{as }p\to\infty.
$$
\cite{gao2018fundamental} showed that a natural solution to this problem can be obtained using the
concentration of maxima phenomenon.  Specifically, consider the class of all thresholding support estimators:
\begin{equation}\label{def: S_p}
\widehat S_p:= \{ j\in [p]\, :\, x_p(j) > t_p(x)\},    
\end{equation}
where $t_p(x)$ is possibly data-dependent threshold. For simplicity of exposition, suppose also that the signal magnitude is parametrized as follows \begin{equation*}
\mu_p(i)=\sqrt{2r\log(p)},\quad i\in S_p,   
\end{equation*} where $r>0.$ Consider also the function $$g(\beta):=(1+\sqrt{1-\beta})^2.$$ Theorems 2.1 and 2.2 of \cite{gao2018fundamental} entail that if $\mathcal{E}$ is URS (see Definition \ref{def: URS} below), then we have the phase-transition:
\begin{equation*}
{\mathbb P}[\what S_p = S_p] \longrightarrow  \left\{ \begin{array}{ll}
 1, & \mbox{if $r> g(\beta)$ for suitable } \what S_p\ \mbox{as in \eqref{def: S_p}}\\
 0, & \mbox{if $r< g(\beta)$ for all }\what S_p\ \mbox{as in \eqref{def: S_p}}
 \end{array}\right.,\quad \mbox{as } p\to\infty.    
\end{equation*}

That is, for signal magnitudes above the boundary, thresholding (Bonferonni-type) estimators recover the support perfectly, as $p\to\infty$;
whereas for signals below the boundary, no thresholding estimators can recover the support with positive probability.   Further, as shown in
\cite{gao2018fundamental}, thresholding estimators are optimal in the iid Gaussian setting and hence the above 
phase-transition applies to {\em all possible} support estimators leading to minimax-type results.  Interestingly, 
both Gaussian and non-Gaussian noise arrays are addressed equally well, provided that 
they satisfy the uniform relative stability property.
While URS is a very mild condition, except for the Gaussian case addressed in \cite{gao2018fundamental}, little is known in general. Here, we will fill this gap for a class of functions of Gaussian arrays (see Section \ref{sec: functions Gaussian}), using our new results on the rates of concentration. 

\subsection{Concentration of maxima} \label{sec:gao-review}

In this section, we recall some definitions and a characterization of URS in
\cite{gao2018fundamental}. We start by presenting the notion of \textit{relative stability}.  

\begin{definition}\label{def: RS}
(Relative stability). Let $\epsilon_p=(\epsilon_p(j))_{j=1}^p$ be a sequence of random variables with identical marginal distributions $F$. Define the sequence $(u_p)_{p=1}^{\infty}$ to be the $(1-1/p)$-th quantile of $F$, i.e., 
\begin{equation}\label{def: up}
u_p=F^{\leftarrow}(1-1/p).    
\end{equation}
The triangular array $\mathcal{E}=\{\epsilon_p,\ p\in\mathbb{N}\}$ is said to have relatively stable (RS) maxima if  \begin{equation}\label{def: RS limit}
\frac{1}{u_p}M_p:=\frac{1}{u_p}\max_{i=1,\hdots,p}\epsilon_p(i)\stackrel{{\mathbb P}}{\to}1,    
\end{equation}
as $p\to\infty.$
\end{definition}
Note that by Proposition 1.1 of \cite{gao2018fundamental}, we have for the standard Normal distribution, that \begin{equation}\label{eq: quantiles}
u_p = \Phi^{\leftarrow}(1-1/p)\sim \sqrt{2\log(p)}.   
\end{equation} 
While relative stability is not directly used in this paper, it is a natural prerequisite to introducing the following generalization. 
\begin{definition}\label{def: URS}
(Uniform Relative Stability (URS)). Under the notations established in Definition \ref{def: RS}, the triangular array $\mathcal{E}=\left\{\epsilon_p(i),\ i\in[p]\right\}$ is said to have uniform relatively stable (URS) maxima if for $\textit{every}$ sequence of subsets $S_p\subseteq\{1,\hdots,p\}$ such that $|S_p|\to\infty,$ we have \begin{equation}\label{def: URS lim}
\frac{1}{u_{|S_p|}}M_{S_p}:= \frac{1}{u_{|S_p|}}\max_{i\in S_p}\epsilon_p(i)\stackrel{{\mathbb P}}{\longrightarrow}1, \quad \mbox{as } p\to\infty.   
\end{equation}
\end{definition}

\begin{definition}\label{def: UDD}
(Uniformly Decreasing Dependence (UDD)).A Gaussian triangular array $\mathcal{E}$ with standard normal marginals is said to be uniformly decreasingly dependent (UDD) if for every $\tau>0$ there exists a finite $N_{\mathcal{E}}(\tau)<\infty,$ such that for every $i\in\{1,\hdots,p\},$ and $p\in\mathbb{N}$, we have \begin{equation}\label{def: Ntau}
\Bigl\lvert \{k\in\{1,\hdots,p\}:\cov(\epsilon_p(k),\epsilon_p(i))>\tau\}\Bigr\rvert\leq N_{\mathcal{E}}(\tau),\quad \text{for all}\ \tau>0.    
\end{equation} 
That is, for any coordinate $j$, the number of coordinates which are more than $\tau$-correlated with $\epsilon_p(j)$ does not exceed $N_{\mathcal{E}}(\tau).$
\end{definition}
The next result provides the equivalence between uniform relative stability and uniformly decreasing dependence.

\begin{theorem}[Theorem 3.2 in \cite{gao2018fundamental}]\label{thm: UDD equiv URS}
Let $\mathcal{E}$ be a Gaussian triangular array with standard Normal marginals. The array $\mathcal{E}$ 
is URS if and only if it is UDD.
\end{theorem}

Theorem \ref{thm: UDD equiv URS} is the starting point of the rate investigations in our paper.  Our main result,
Theorem \ref{thm: capstone}, below, extends the former by providing upper bounds on the rate of concentration. Before that, though, in Section \ref{sec:optimal} we study cases where the optimal rate can be formally established.

\begin{remark}[{\it On our use of the term ``upper bound''}]
Fix a positive sequence $\delta_p^{\star}\downarrow0.$ We refer to $\delta_p^{\star}$ as an upper bound on the rate of concentration when \eqref{e:rate-delta_p} holds for any sequence $\delta_p\gg \delta_p^{\star}$. Further, for two positive sequences $\alpha_p$ and $\beta_p$ we write $\alpha_p\asymp\beta_p$ if $$0<c_1\leq\liminf_{p\to\infty}\left|\frac{\alpha_p}{\beta_p}\right|\leq\limsup_{p\to\infty}\left|\frac{\alpha_p}{\beta_p}\right|\leq c_2<\infty.$$

Let $\delta_p^{\star}$ be an upper bound on the rate of concentration and $\delta_p\gg\delta_p^{\star}.$ Then, naturally, \eqref{e:rate-delta_p} holds with $\delta_p$ replaced by $\tilde\delta_p,$ for any $\tilde\delta_p\asymp\delta_p.$
\end{remark}
\medskip{}

\section{On the optimal rate of concentration} \label{sec:optimal}

In this section, we provide some general comments on the fastest possible rates of concentration for 
maxima of Gaussian variables.  Somewhat surprisingly, the rate depends on the choice of the normalizing
sequence $u_p$.  As it turns out poor choices of normalizing sequences can lead to arbitrarily slow rates. On the other hand, 
for a wide range of dependence structures (including the iid case), the best possible rate will be shown to be $1/\log(p)$. 
The question of whether the maxima of dependent Gaussian arrays can concentrate faster than that rate, however unlikely 
this may be, is open, to the best of our knowledge (cf Conjecture
\ref{con: optimal rate}, below).

Consider a Gaussian array ${\cal E}=\{\epsilon_{p}(i),\ i\in [p]\}$ with standard Normal marginal.  We shall assume that 
${\cal E}$ is (uniformly) relatively stable, so that in particular,
$$
\frac{1}{u_p} \max_{i\in [p]} \epsilon_p(i) =: \frac{M_p}{u_p} \stackrel{P}{\longrightarrow} 1,
$$
as $p\to\infty$, where $u_p := \Phi^{-1}(1-1/p)$ is the $(1/p)$-th tail quantile of the standard Normal distribution.

We consider the iid case first and, for clarity, let $M_p^*$ denote the maximum of $p$ independent standard Normal 
random variables.  Suppose that for some $a_p>0$ and $a_p,\ b_p\in \mathbb R$, we have
$$
\Phi(a_p^{-1} x + b_p)^p \to \Lambda(x):= \exp\{-e^{-x}\},\quad\mbox{as }p\to\infty,
$$
for all $x\in\mathbb R$.  That is, we have
\begin{equation}\label{e:iid-max-to-Gumbel}
a_p( M_p^{*} - b_p) \stackrel{d}{\longrightarrow} \zeta,\ \ \mbox{ as }p\to\infty,
\end{equation}
where $\zeta$ has the standard Gumbel distribution $\Lambda$.  The next result is well-known.  We give it here since it 
summarizes and clarifies the possible choices of the normalizing constants $a_p$ and $b_p$ for \eqref{e:iid-max-to-Gumbel} to hold.

\begin{lemma}\label{l:ap-bp-constants} {\em (i)} We have that
\begin{equation}\label{e:up-equivalence}
\widetilde u_p (M_p^* - \widetilde u_p) \stackrel{d}{\longrightarrow} \zeta\quad\mbox{ if and only if}\quad p \overline\Phi(\widetilde u_p) \to 1,  
\end{equation}
as $p\to\infty$.  In this case, $\widetilde u_p \sim \sqrt{2\log(p)}$ and more precisely 
\begin{equation}\label{e:up-up-star}
\sqrt{2\log(p)}( \widetilde u_p - u_p^*) \to 0,\ \ \mbox{ as }p\to\infty,
\end{equation}
where
\begin{equation}\label{e:up-tilde}
  u_p^* := \sqrt{2\log(p)} \left( 1- \frac{\log(\log(p)) + \log(4\pi)}{4 \log(p)} \right).
\end{equation}

{\em (ii)} Relation \eqref{e:iid-max-to-Gumbel} holds if and only if
$$
a_p \sim \sqrt{2\log(p)}\quad\mbox{ and }\quad p\overline \Phi(b_p)\to 1.
$$
In particular, by part (i), we have that \eqref{e:iid-max-to-Gumbel} holds with $a_p:=b_p$ and \eqref{e:up-up-star} holds
with $\widetilde u_p:=b_p$.
\end{lemma}
\begin{proof} {\em Part (i).} Observe that by the Mill's ratio (cf Lemma \ref{l:Mills_ratio}),
$p \overline \Phi(\widetilde u_p) \to 1$ is equivalently expressed as follows:
$$
p \overline \Phi(\widetilde u_p) \sim p \frac{\phi(\widetilde u_p)}{\widetilde u_p} \to 1,\quad\mbox{as }p\to\infty,
$$
where $\phi(x) = \exp\{-x^2/2\}/\sqrt{2\pi}$ is the standard Normal density. By taking logarithms,
the above asymptotic relation is equivalent to having
\begin{equation}\label{e:log-quantile-relation}
\log(p) - \frac{\widetilde{u_p}^2}{2} - \log(\widetilde u_p) - \frac{1}{2}\log(2\pi)\to 0.
\end{equation}

We first prove the `if' direction of part (i).  Suppose that $p\overline \Phi(\widetilde u_p)\to 1$, or equivalently, 
\eqref{e:log-quantile-relation} holds. Then, one necessarily has $\widetilde u_p\to \infty$. It is easy to see that 
\eqref{e:iid-max-to-Gumbel} holds with $a_p:= \widetilde u_p$ and 
$b_p:= \widetilde u_p,$ provided that,  for all $x\in\mathbb R$,
\begin{equation}\label{e:up-limit}
\Phi\left( \widetilde u_p + \frac{x}{\widetilde u_p}\right)^p \to \Lambda(x),\quad\mbox{as }p\to\infty.
\end{equation}
The latter, upon taking logarithms and using the fact that $\log(1+z) \simeq z,$ as $z\to 0$,
is equivalent to having
\begin{equation}\label{e:iid-max-to-Gumbel-1}
p \overline \Phi\left( \widetilde  u_p + \frac{x}{\widetilde  u_p}\right) \to -\log (\Lambda(x)) = e^{-x}.
\end{equation}
To prove that \eqref{e:iid-max-to-Gumbel-1} holds, as argued above, using the Mill's ratio, it is 
equivalent to verify that
$$
A_p:= \log(p) - \frac{1}{2} \left( \widetilde  u_p + x/{\widetilde  u_p}\right)^2 
 - \log\left( \widetilde  u_p + x/{\widetilde  u_p}\right) - \frac{1}{2}\log(2\pi) \to -x,
$$
as $p\to\infty$.  Note that, upon expanding the square and manipulating the logarithm, we obtain
\begin{align*}
A_p = \log(p) - \frac{\widetilde {u_p}^2}{2} - \log(\widetilde  u_p) - \frac{1}{2} \log(2\pi)  - x - x^2/(2\widetilde{u_p}^2) - \log( 1+ x/\widetilde{u_p}^2). 
\end{align*}
In view of \eqref{e:log-quantile-relation} and the fact that $\widetilde{u_p}\to \infty$, we obtain that $A_p\to -x$, which yields 
\eqref{e:iid-max-to-Gumbel-1} and completes the proof of the `if' direction of part (i).

Now, to show the `only if' direction of part (i), suppose that \eqref{e:iid-max-to-Gumbel} holds with $a_p=b_p:= \widetilde u_p$, 
or, equivalently \eqref{e:up-limit} holds.  By letting $x=0$ in Relation \eqref{e:up-limit}, we see that $\widetilde u_p\to \infty$, 
and then, upon taking logarithms, necessarily $p\overline \Phi(\widetilde u_p)\to 1$, which completes the proof of \eqref{e:up-equivalence}.

\medskip
We now show \eqref{e:up-up-star}.  First, one can directly verify that \eqref{e:log-quantile-relation} holds with $\widetilde u_p$ replaced 
by $u_p^*$ in \eqref{e:up-tilde}. This, as argued above, is equivalent to $p \overline \Phi(u_p^*)\to 1$.  
Suppose now that, for another sequence $\widetilde u_p$, we have $p \overline \Phi(\widetilde u_p)\to 1$. Then, by the shown equivalence in \eqref{e:up-equivalence},
$$
u_p^*(M_n^* - u_p^*) \stackrel{d}{\to} \zeta \quad\mbox{ and }\quad \widetilde u_p (M_n^* - \widetilde u_p) \stackrel{d}{\to} \zeta.
$$
Thus, the convergence of types theorem (see, \textit{e.g.},Theorem 14.2 in \cite{Billingsley1995-BILMAP}) yields
$$
u_p^*\sim \widetilde u_p \quad \mbox{ and }\quad u_p^*( u_p^* - \widetilde u_p) \to 0.
$$
The last convergence implies the claim of part (ii) since in view of \eqref{e:up-tilde}, we have $u_p^*\sim \sqrt{2\log(p)}$.

Part (ii) is a direct consequence of the convergence to types theorem, as argued in the proof of part (i).\qed
\end{proof}
\medskip

The following result characterizes the optimal rate of concentration under an additional distributional convergence assumption, which holds under the Berman condition for e.g. the case of stationary time series.

\begin{proposition}\label{prop: on optimal rate}
Suppose that ${\cal E}$ is a {\em dependent} triangular Gaussian array, such that 
\begin{equation}\label{e:dependent-max}
\zeta_p:= a_p(M_p - b_p) \stackrel{d}{\longrightarrow}\zeta,\quad\mbox{as }p\to\infty,
\end{equation}
for some non-degenerate random variable $\zeta$, with the same constants as in the iid case \eqref{e:iid-max-to-Gumbel}. Suppose also that ${\mathbb P}\left(\zeta<x\right)>0$ and ${\mathbb P}\left(\zeta>x\right)>0$ for all $x\in\mathbb{R}.$ 

Let now the sequence $\delta_p\to 0$, be an upper bound on the
rate of concentration, i.e., we have 
\begin{equation}\label{def: opt rate URS lim}
{\mathbb P}\left(\left|\frac{M_p}{a_p}-1\right|>\delta_p\right)\to 0, \quad p\to\infty.
\end{equation}

\noindent The following two statements hold.

\begin{itemize}
    \item[(a)] When $\limsup_{p\to\infty}a_p|b_p-a_p|<\infty,$ Relation  \eqref{def: opt rate URS lim} holds if and only if 
    \begin{equation}\label{def: delta_p opt}
   \delta_p\gg \frac{1}{a_p^2}+\left|\frac{b_p}{a_p}-1\right|=:\delta_p^{opt}.   
    \end{equation}
    
    \item[(b)] When $\limsup_{p\to\infty}a_p|b_p-a_p|=\infty,$ Relation \eqref{def: opt rate URS lim} holds if and only if 
    \begin{equation}\label{prop: on optimal rate: liminf}
    \liminf_{p\to\infty}\left[\frac{\delta_p}{\delta_p^{opt}}-1\right](1+a_p|b_p-a_p|)=\infty.    
    \end{equation}
\end{itemize}
\end{proposition}
\begin{proof}
{{\em (a)}}
We will start with the ``if'' direction. Relation \eqref{e:dependent-max} implies that
$$
\frac{1}{a_p} M_p = \frac{\zeta_p}{a_p^2} + \frac{b_p}{a_p}.
$$
Since by assumption the constants $a_p$ and $b_p$ are the same as in the iid case \eqref{e:iid-max-to-Gumbel}, Lemma \ref{l:ap-bp-constants} entails that $b_p\sim a_p \sim \sqrt{2\log(p)}.$ Hence
\begin{equation}\label{e:ap-rate}
\frac{1}{a_p} M_p -1 = \frac{\zeta_p}{a_p^2} + \left( \frac{b_p}{a_p} -1\right) \to 0,
\end{equation}
which shows that the distributional limit in \eqref{e:dependent-max} entails concentration of the maxima $M_p/a_p$ to $1$. Relations \eqref{def: delta_p opt} and \eqref{e:ap-rate}, however imply that $$\left|\frac{M_p}{a_p}-1\right|=o_P(\delta_p),$$
which entails \eqref{def: opt rate URS lim} by Slutsky  (or also Lemma \ref{l:fastest-rate}, below.)

Now, for the converse direction, suppose that \eqref{def: opt rate URS lim} holds for some $\delta_p\not\gg\delta_p^{opt}.$ This means that we can find a subsequence $p(n)$ so that $\delta_{p(n)}\leq c\cdot\delta_{p(n)}^{opt},\ \forall n\in\mathbb{N}$, for a positive constant $c$ that does not depend on $n$. In view of \eqref{def: opt rate URS lim}, this would mean that 
$$\theta_n:={\mathbb P}\left(\left|\frac{M_{p(n)}}{a_{p(n)}}-1\right|>c\delta_{p(n)}^{opt}\right)\to0,\quad n\to\infty.$$
Moreover, since $\limsup_{p\to\infty}a_p|b_p-a_p|<\infty,$ and $a_p>0$, the sequence $(a_p|b_p-a_p|)_{p=1}^{\infty}$ is bounded. Namely, there exists $M>0$, such that   $0\leq a_p|b_p-a_p|\leq M,$ for all $p\in\mathbb{N}.$
However, we have that 
\begin{align*}
\theta_n & \geq {\mathbb P}\left(\frac{M_{p(n)}}{a_{p(n)}}-1>c\delta_{p(n)}^{opt}\right) = {\mathbb P}\left(\frac{\zeta_{p(n)}}{a_{p(n)}^2}+\frac{b_{p(n)}}{a_{p(n)}}-1>\frac{c}{a_{p(n)}^2}+c\left|\frac{b_{p(n)}}{a_{p(n)}}-1\right|\right)\\
& = {\mathbb P}\left(\zeta_{p(n)}+a_{p(n)}(b_{p(n)}-a_{p(n)})-c|a_{p(n)}(b_{p(n)}-a_{p(n)})|>c\right)\\
& \geq {\mathbb P}\left(\zeta_{p(n)}-(c+1)a_{p(n)}|b_{p(n)}-a_{p(n)}| >c\right)\\ & \geq {\mathbb P}\left(\zeta_{p(n)}>c+(c+1)a_{p(n)}|b_{p(n)}-a_{p(n)}|\right)\\ & \geq {\mathbb P}(\zeta_{p(n)}>c+(c+1)M)\\
& \to {\mathbb P}(\zeta>c+(c+1)M)>0, 
\end{align*}
where the last convergence holds because $\zeta_{p(n)}\stackrel{d}{\to}\zeta$. This is a contradiction and the proof is complete.\\

{{\em (b)}} We have that \begin{align*}
{\mathbb P}\left(\left|\frac{M_p}{a_p}-1\right|>\delta_p\right) & = {\mathbb P}\left(a_p\left|M_p-a_p\right|>\delta_p a_p^2\right) = {\mathbb P}\left(\left|\zeta_p+a_p(b_p-a_p)\right|>\delta_pa_p^2\right)\\ & = {\mathbb P}\left(\zeta_p<-\delta_pa_p^2-a_p(b_p-a_p)\right)+{\mathbb P}\left(\zeta_p>\delta_pa_p^2-a_p(b_p-a_p)\right)\\ & =: A(p)+B(p).    
\end{align*}
Note, however, that \eqref{def: opt rate URS lim} entails that both $A(p)$ and $B(p)$ vanish to 0, as $p\to\infty.$
This in turn means that \begin{equation}\label{eq: liminf }
\liminf_{p\to\infty}(\delta_pa_p^2-a_p(b_p-a_p))=\infty\quad\text{and}\quad \liminf_{p\to\infty}(\delta_pa_p^2+a_p(b_p-a_p))=\infty,
\end{equation} 
because of the distributional convergence \eqref{e:dependent-max}. We will work with $B(p)$. The result for $A(p)$ can be obtained by similar arguments. At first, for $B(p)$ to vanish to 0, we do need $\delta_pa_p^2>a_p(b_p-a_p)$ eventually. Suppose that $\liminf_{p\to\infty}(\delta_pa_p^2-a_p(b_p-a_p))=c<\infty,$ where $c\geq0$. This would mean that there is a subsequence $p(n)$ such that $$\delta_{p(n)}a_{p(n)}^2-a_{p(n)}(b_{p(n)}-a_{p(n)})\to c,\quad p\to\infty.$$
But then, $$B(p(n))={\mathbb P}\left(\zeta_{p(n)}>\delta_{p(n)}a_{p(n)}^2-a_{p(n)}(b_{p(n)}-a_{p(n)})\right)\to{\mathbb P}(\zeta>c)>0,$$
which contradicts the fact that $B(p)\to 0,$ as $p\to\infty.$\\ Finally, note that \eqref{eq: liminf } is equivalent to $\liminf_{p\to\infty}(\delta_p a_p^2-a_p|b_p-a_p|)=\infty,$ which with straightforward algebra can be expressed as \eqref{prop: on optimal rate: liminf}. Indeed, 
\begin{align*}
\delta_pa_p^2-a_p|b_p-a_p| & = a_p^2\left[\delta_p -\left|\frac{b_p}{a_p}-1\right|\right]= a_p^2\left[\delta_p-\delta_p^{opt}\right] +1\\ & = a_p^2\delta_p^{opt}\left[\frac{\delta_p}{\delta_p^{opt}}-1\right]+1\\ &= \left[\frac{\delta_p}{\delta_p^{opt}}-1\right]\left(1+a_p|b_p-a_p|\right)+1,
\end{align*}
which completes the proof. \qed
\end{proof}

\begin{remark}[On the optimality of the rate $\delta_p^{opt}$] \label{re: optimal rate  delta_opt}
The rate $\delta_p^{opt}$ can be viewed as ``the'' optimal rate of concentration in \eqref{def: opt rate URS lim} in the sense of \eqref{def: delta_p opt} and \eqref{prop: on optimal rate: liminf}.  As pointed out by an anonymous referee, the distributional convergence in \eqref{e:dependent-max} 
(whenever it takes place) is much more informative than a simple concentration of maxima type convergence.  Specifically, by Lemma 
\ref{l:ap-bp-constants} (ii), one can take $u_p = a_p = b_p$, and in this case Relation \eqref{e:ap-rate} implies that
$1/a_p^2 \propto 1/\log(p)$ is both an upper and lower bound on the rate of concentration.  That is, the rate 
$\delta_p^{opt} = 1/a_p^2 \propto 1/\log(p)$ cannot be improved and in this sense is the optimal rate at which the maxima can concentrate.   
The rate of concentration, though, does depend on the choice of the normalization sequence $u_p$.  We elaborate on this point next.
\end{remark}

{\bf The role of the sequence $u_p$.} 
It is well-known that under quite substantial dependence, the  convergence in distribution \eqref{e:dependent-max} holds, with the 
{\em same constants} as in the independent case. For example, suppose that $\epsilon_p(i) = Z(i),\ i\in \mathbb Z$ come from a stationary 
Gaussian time series, which satisfies the so-called 
{\em Berman condition} \citep{berman1964limit}:
$$
{\rm Cov}(Z(k), Z(0)) = o\left( \frac{1}{\log(k)} \right),\quad\mbox{as } k\to\infty.
$$

Notice, by Lemma \ref{l:ap-bp-constants} (ii), however, we also have $\widetilde \zeta_p := b_p(M_p - b_p) \stackrel{d}{\to} \zeta$, and
\begin{equation}\label{e:bp-rate}
\frac{1}{b_p} M_p - 1 = \frac{\widetilde \zeta_p}{b_p^2} = {\cal O}_P\left( \frac{1}{\log(p)} \right).
\end{equation}

Compare Relations \eqref{e:ap-rate} and \eqref{e:bp-rate}.  Since $a_p\sim b_p \sim \sqrt{2\log(p)}$, from \eqref{e:bp-rate}, 
we have that the rate of concentration of $M_p$ relative to the sequence $b_p$ is $1/\log(p)$. On the other hand, while the first term in 
the right-hand side of \eqref{e:ap-rate} is of order $1/\log(p)$ the presence of the second term can only make the rate of concentration therein {\em slower}.
Indeed, this is formally established in Lemma \ref{l:fastest-rate}. To gain some more intuition that the poor choice of a sequence $a_p$ can lead to a
slower rate of concentration, suppose that $a_p = b_p/(1+g(p))$, for an arbitrary sequence $g(p)>-1$, such that $g(p)\to 0$.  Then, by \eqref{e:ap-rate},
$$
\frac{1}{a_p} M_p -1 = \frac{\zeta_p}{a_p^2} + g(p).
$$
One can take $g(p)\to 0$ arbitrarily slow.  Finally, as a more concrete example, one typically uses $b_p:= u_p^* = \sqrt{2\log(p)} (1 - (\log(\log(p)) + \log(4\pi))/4\log(p))$ and $a_p:= \sqrt{2\log(p)}$. It is easily seen that $b_p = a_p (1+g(p))$, where
$$
g(p) = - \frac{\log(\log(p)) + \log(4\pi)}{4\log(p)} \propto \frac{\log(\log(p))}{\log(p)}.
$$
This shows that, in particular, in the case of iid maxima (as well as in the general case where \eqref{e:dependent-max} holds) 
the normalization $\sqrt{2\log(p)}$ {\em does not lead} to the optimal rate, since
$$
\frac{1}{\sqrt{2\log(p)}} M_p^* -1 \propto_P \frac{\log(\log(p))}{\log(p)},
$$
where $\xi_p \propto_P \eta_p$ means that $\xi_p/\eta_p \to c$ in probability, for some positive constant $c$.

The optimal rate is $1/\log(p)$ and it is obtained by normalizing with any sequence $b_p$ such that $p\overline \Phi(b_p) \to 1$. This follows from the next simple result, which shows that the rate of concentration in \eqref{e:ap-rate} is the slower of the rates
$1/a_p^2$ and $(b_p-a_p)/a_p$. 

\begin{lemma}\label{l:fastest-rate} Suppose that for some random variables $\zeta_p$, we have 
$\zeta_p\stackrel{d}{\to} \zeta$, as $p\to\infty$, where $\zeta$ is a non-constant random variable.
Then, for all sequences $\alpha_p$ and $\beta_p$, we have $$
\alpha_p \zeta_p + \beta_p\stackrel{P}{\longrightarrow} 0\quad\Longleftrightarrow\quad |\alpha_p| + |\beta_p| \longrightarrow 0.
$$

That is, the {\em rate} of $\alpha_p \zeta_p + \beta_p$ is always the slower of the rates of $\{\alpha_p\}$ and $\{\beta_p\}$.
\end{lemma}
\begin{proof} The '$\Leftarrow$' direction follows from Slutsky. To prove '$\Rightarrow$', it is enough to show that for every $p(n)\to \infty$, there is a further sub-sequence $q(n)\to \infty$, $\{q(n)\}\subset\{p(n)\}$, such that
$$
|\alpha_{q(n)}| + |\beta_{q(n)}| \longrightarrow 0.
$$
In view of Skorokhod's representation theorem (Theorem 6.7, page 70 in \cite{billingsley2013convergence}), we may suppose that $\zeta_p^* \to \zeta^*$, with probability one, where $\zeta_p^* \stackrel{d}{=}\zeta_p$ and
 $\zeta^* \stackrel{d}{=}\zeta$.  Also, assuming that $\alpha_{p(n)} \zeta_{p(n)}^* + \beta_{p(n)} \to 0$, in probability,
 implies that there is a further sub-sequence $q(n)\to\infty$, such that 
 \begin{equation}\label{e:l:fastest-rate}
  \alpha_{q(n)}  \zeta_{q(n)}^*(\omega) + \beta_{q(n)} \to 0,\quad\mbox{as }q(n)\to\infty,
 \end{equation}
 for $P$-almost all $\omega$. Since also $\zeta_{q(n)}^*(\omega) \to \zeta^*(\omega)$, for $P$-almost all $\omega$, and since $\zeta^*$ is non-constant, 
 we have $\zeta_{q(n)}^*(\omega_i) \to \zeta^*(\omega_i),\ i=1,2$ for some $\zeta^*(\omega_1)\not= \zeta^*(\omega_2)$.

 Thus, by subtracting two instances of Relation \eqref{e:l:fastest-rate} corresponding to $\omega=\omega_1$ and $\omega=\omega_2$, we obtain
 $$
 \alpha_{q(n)}  (\zeta_{q(n)}^*(\omega_1) - \zeta_{q(n)}^*(\omega_2)) \to 0,
 $$
 which since $(\zeta_{q(n)}^*(\omega_1) - \zeta_{q(n)}^*(\omega_2))\to \zeta^*(\omega_1) - \zeta^*(\omega_2) \not = 0$, implies
 $\alpha_{q(n)} \to 0$.  This, in view of \eqref{e:l:fastest-rate} yields $\beta_{q(n)}
 \to 0$, and completes the proof. \qed
\end{proof}

\medskip
\noindent
\begin{remark}\label{rem: conjecture}
The above considerations establish the optimal rate of concentration of the maxima $M_p = \max_{i\in[p]} \epsilon_p(i)$, 
whenever the limit in distribution \eqref{e:dependent-max} holds.  We have shown that this optimal rate is $1/\log(p)$ and is in fact obtained, 
when considering $M_p/u_p$, for $p\overline \Phi(u_p) \sim 1$.  The rate of concentration of $M_p/\sqrt{2\log(p)}$ is $\log(\log(p))/\log(p)$, which is 
only slightly sub-optimal. 

On the other hand, as we know by Theorem \ref{thm: UDD equiv URS}, uniform relative stability is equivalent to UDD and hence
the concentration of maxima phenomenon takes place even if \eqref{e:dependent-max} fails to hold.  At this point, we do not know what is the optimal 
rate in general. In Section \ref{sec:main}, we provide upper bounds on this rate.  We conjecture, however, the presence of more severe
dependence can only lead to slower rates of concentration and in particular the optimal rate of concentration for UDD arrays cannot be faster 
than $1/\log(p)$ -- the one for independent maxima.
\end{remark} 

\begin{conjecture}\label{con: optimal rate}
Let $\mathcal{E}$ be a Gaussian URS array. Relation \eqref{def: URS with delta_p} implies $\delta_p\gg 1/\log(p).$
\end{conjecture}

\section{Rates of uniform relative stability} \label{sec:main}

\subsection{Gaussian arrays}
Throughout Sections \ref{sec:main} and \ref{sec:proofs}, $\mathcal{E}=\{\epsilon_p(i),\ i\in[p]\}$ will be a Gaussian array with standard Normal marginals, unless stated otherwise. We shall also assume that $\mathcal{E}$ is URS. For simplicity of notation and without loss of generality we will work with $S_p=[p]$ (see Remark \ref{rem: WLOG |S_p|=p}). We will obtain upper bounds on the rate, i.e., sufficient conditions on the dependence structure of $\mathcal{E}$, which ensure certain rates. These results are of independent interest and will find concrete applications in Section \ref{sec: functions Gaussian}, where conditions ensuring the URS of functions of Gaussian arrays are established. 

The following definition is an ancillary tool for the comparison of the rates of two vanishing sequences and introduces some notation for this purpose. 
\begin{definition}\label{def: slower rate}
Let $(\alpha_p)_{p=1}^{\infty}$ and $(\beta_p)_{p=1}^{\infty}$ be two positive sequences converging to 0. We will say that  $\alpha_p$ is of lower order than $\beta_p$ (or slower than $\beta_p$) ,  denoted by $\alpha_p\gg \beta_p$, if $\beta_p/\alpha_p\to 0,$ as $p\to\infty$, i.e., $\beta_p=o(\alpha_p)$.
\end{definition}

The next theorem constitutes the main result of this paper.

\begin{theorem}\label{thm: capstone}
Consider a UDD Gaussian triangular array $\mathcal{E}=\{\epsilon_p(i),i\in[p]\}$  with standard Normal marginals and let $N_{\mathcal{E}}(\tau)$ be as in Definition \ref{def: UDD}. Let $\tau(p)\to0$ be such that \begin{equation}\label{def: alpha(p)}
    \alpha(p):=\log N_{\mathcal{E}}(\tau(p))/\log (p)\to0,\quad{\rm as}\ p\to\infty.
\end{equation}
Then, for all $\delta_p>0$ such that \begin{equation}\label{ineq: lb on the rate}
\delta_p\gg \alpha(p)+\tau(p)+\frac{1}{\log(p)},
\end{equation}
we have 
\begin{equation}\label{def: URS with delta_p}
{\mathbb P}\left(\left|\frac{\max_{i\in[p]}\epsilon_p(i)}{u_p}-1\right|>\delta_p\right)\to 0, \quad\mbox{as } p\to\infty.
\end{equation}
Here $u_p$ is defined as in \eqref{def: up} taking $F=\Phi$, the cumulative distribution function of standard Normal distribution.
\end{theorem}

The proof of Theorem \ref{thm: capstone} depends on a number of technical results, which will be presented and proved in Section \ref{sec:proofs}. In order to make the proof easier for the reader to follow,  we postpone its demonstration until Section \ref{sec:proofs}.
We proceed next with several comments and examples.

\begin{remark}\label{rem: N_epsilon tau uniform bound} 
Note that in Theorem \ref{thm: capstone} the covariance structure of $\mathcal{E}$ appears only through $N_{\mathcal{E}}(\tau).$ The collection  $\left\{N_{\mathcal{E}}(\tau),\ \tau\in (0,1)\right\}$ constitutes a collection of uniform upper bounds on the number of covariances in each row of the triangular array $\mathcal{E}$ that exceed the threshold $\tau$. This means that the ordering of the $p$ random variables in each row of $\mathcal{E}$ is irrelevant. 
\end{remark}

\begin{remark}\label{rem: WLOG |S_p|=p}
The support recovery results of \cite{gao2018fundamental} require URS in the sense of \eqref{def: URS lim} for a subsequence $S_p\subset [p]$, with $|S_p|\to\infty.$ By the previous remark, upon relabelling the triangular array $\mathcal{E}$, Theorem \ref{thm: capstone} applies in this setting with $p$ replaced by $|S_p|$, and entails rates on the convergence in \eqref{def: URS lim}. 
\end{remark}

The preceding Theorem \ref{thm: capstone} gives us an upper bound on the rate at which the convergence in \eqref{def: URS lim} takes place for a UDD Gaussian array $\mathcal{E}.$ Observe that this bound depends  crucially on the covariance structure of $\mathcal{E}$ through $N_{\mathcal{E}}(\tau)$. This dependence will be illustrated in the following examples, where the upper bound stated in \eqref{ineq: lb on the rate} is obtained for three specific covariance structures.
\begin{example}\label{ex: iid Gaussian}
\textit{The iid case and optimality of the rate bounds.}\\
Suppose that all $\epsilon_p(j)$'s are iid. Then, we can pick $\tau(p)=0$ or $\tau<1$ vanishing to 0 arbitrarily fast, and we would have that $N_{\mathcal{E}}(\tau)=1,$ because of the strict inequality in \eqref{def: Ntau}. This implies that $\alpha(p)=\log(N_{\mathcal{E}}(\tau))/\log(p)=0.$ Thus, in this case, the upper bound in \eqref{ineq: lb on the rate} becomes $1/\log(p)$. Observe that this rate matches the optimal rate in Conjecture \ref{con: optimal rate}.
\end{example}

\begin{example}\label{ex: cov k^-gamma}
\textit{Power-law covariance decay.}\\
Consider, first, the simple case where $\mathcal{E}$ comes from a stationary Gaussian time series, $\epsilon_p(\kappa)=\epsilon(\kappa)$, with auto-covariance \begin{equation}\label{def: power-law cov}
\rho(\kappa)=\cov(\epsilon(\kappa),\epsilon(0))\propto \kappa^{-\gamma},\quad \gamma>0.    
\end{equation}
Then, the classic Berman condition $\rho(\kappa)=o(1/\log(\kappa))$ holds and as shown in  the discussion after Proposition \ref{prop: on optimal rate},  the optimal rate in \eqref{def: URS lim} is $1/\log(p)$.

In this example, we will demonstrate that our result [Theorem \ref{thm: capstone}] leads to the nearly optimal rate $\log(\log(p))/\log(p)$. As in the previous remark, we see that this is in fact the optimal rate if $u_p$ in \eqref{def: URS lim} is replaced by $\sqrt{2\log(p)}.$ (See Section \ref{sec:optimal}). Note, however, that our arguments apply in greater generality and do not depend on the stationarity assumption. Indeed, assume that $\mathcal{E}$ is a general Gaussian triangular array such that (UDD$^\prime$) of \cite{gao2018fundamental} holds, i.e., \begin{equation}\label{ineq: power law cov bound}
\left|\cov(\epsilon_p(i),\epsilon_p(j))\right|\leq c\left|\pi_p(i)-\pi_p(j)\right|^{-\gamma}    
\end{equation}
for suitable permutations $\pi_p$ of $\{1,\hdots,p\}$, where $c$ does not depend on $p$. (Note that \eqref{ineq: power law cov bound} entails \eqref{def: power-law cov} for $\pi_p=id$, where $id$ is the identity permutation.) Then, one can readily show that $N_{\mathcal{E}}(\tau)=\mathcal{O}(\tau^{-1/\gamma})$, as $\tau\to0.$
Thus, \begin{equation*}
\alpha(p)=\frac{\log(N_{\mathcal{E}}(\tau))}{\log(p)}\propto\frac{\log(\tau^{-1/\gamma})}{\log(p)}=-\frac{\log(\tau)}{\gamma\log(p)}.    
\end{equation*}
Using this $\alpha(p),$ the upper bound on the rate in  Theorem \ref{thm: capstone} becomes
\begin{equation}\label{eq: alpha_p power law}
    \alpha(p)+\tau(p)+\frac{\log(\log(p))}{\log(p)}\propto -\frac{\log(\tau)}{\log(p)}+\tau(p)+\frac{1}{\log(p)}\asymp-\frac{\log(\tau)}{\log(p)}+\tau(p).
\end{equation}
This is minimized by taking $\tau(p) =1/\log(p)$ in \eqref{eq: alpha_p power law} and the upper bound on the rate becomes $$\alpha(p)+\tau(p)+\frac{1}{\log(p)}\propto \frac{\log(\log(p))}{\log(p)}+\frac{1}{\log(p)}\asymp\frac{\log(\log(p))}{\log(p)}.$$ 
Recall that in the case when ${\cal E}$ has iid components, the
optimal rate of concentration of the maxima is $1/\log(p)$ and in fact
it becomes $\log(\log(p))/\log(p)$ when one uses the normalization
$\sqrt{2\log(p)}$ in place of $u_p$.  Therefore, this example shows 
that under mild power-law type covariance decay conditions, Gaussian 
triangular arrays continue to concentrate at the nearly optimal 
rates for the iid setting.
\end{example}

\begin{example}\label{ex: cov 1/(logk)^nu}
\textit{Logarithmic covariance decay.}\\
Following suit from Example \ref{ex: cov k^-gamma}, we consider first the case where the errors come from a stationary time series with auto-covariance \begin{equation}\label{def: logarithmic cov}
\rho(\kappa)=\cov(\epsilon(\kappa),\epsilon(0))\propto \left(\log(\kappa)\right)^{-\nu},\quad\mbox{as } \kappa\to\infty,   
\end{equation}  
for some $\nu>0$. Note that for $0<\nu<1$, the Berman condition $\rho(\kappa)=o(1/\log(\kappa))$ is no longer satisfied and the results from Section \ref{sec:optimal} cannot be applied to establish the optimal rate in \eqref{def: URS lim}. Using Theorem \ref{thm: capstone}, we will see that an upper bound on this rate is $\delta_p^{\star}:=(\log(p))^{-\frac{\nu}{\nu+1}}$.

Indeed, consider the more general case where $\mathcal{E}$ is a Gaussian triangular array, such that (UDD$'$) of \cite{gao2018fundamental} holds, i.e., \begin{equation}\label{ineq: logarithmic cov bound}
\left|\cov(\epsilon_p(i),\epsilon_p(j))\right|\leq c\left(\log\left(\left|\pi_p(i)-\pi_p(j)\right|\right)\right)^{-\nu},   
\end{equation}
for suitable permutations $\pi_p$ of $\{1,\hdots,p\}$ and $c$ does not depend on $p$. Again, note that \eqref{ineq: logarithmic cov bound} implies \eqref{def: logarithmic cov} for the identity permutation. One can show that in this case $N_{\mathcal{E}}(\tau)=\mathcal{O}\left(e^{\tau^{-1/\nu}}\right)$, as $\tau\to 0$ and thus, $$\alpha(p)=\frac{\log(N_{\mathcal{E}}(\tau))}{\log(p)}\propto\frac{\log\left(e^{\tau^{-1/\nu}}\right)}{\log(p)}=\frac{1}{\tau^{1/\nu}\log(p)},\quad\mbox{as } p\to\infty.$$ 

To find the best bound on the rate in the context of \eqref{ineq: lb on the rate} we minimize \begin{equation*}
\alpha(p)+\tau(p)+\frac{1}{\log(p)}\propto \frac{1}{\tau^{1/\nu}\log(p)}+\tau+\frac{1}{\log(p)},  
\end{equation*}
with respect to $\tau.$ Considering $p$ fixed, basic calculus gives us that the r.h.s. is minimized for $\tau(p)=(\nu\log(p))^{-\frac{\nu}{\nu+1}}.$ With this choice of $\tau$ the fastest upper bound from Theorem \ref{thm: capstone} becomes \begin{equation*}
\left[\nu^{-\frac{\nu}{\nu+1}}+\nu^{-\frac{1}{\nu+1}}\right]\cdot (\log(p))^{-\frac{\nu}{\nu+1}}+\frac{1}{\log(p)}\propto (\log(p))^{-\frac{\nu}{\nu+1}}.    
\end{equation*} 

It only remains to show that the choice of $\tau$ actually allows us to pick  $N_{\mathcal{E}}(\tau)=\mathcal{O}\left(e^{\tau^{-1/\nu}}\right).$ A sufficient condition would be $p\geq \tilde c \cdot e^{\tau^{-1/\nu}}$ for a suitably chosen constant $\tilde c$ not depending on either $p$ or $\tau$. Substituting $\tau=(\nu\log(p))^{-\frac{\nu}{\nu+1}}$, we equivalently need $$p\geq\tilde c \cdot e^{(\nu\log(p))^{\frac{1}{\nu+1}}}.$$ It is readily checked, by taking logarithms in both sides, that this holds for $p$ sufficiently large and thus, the fastest upper bound for this kind of dependence structure is $(\log(p))^{-\frac{\nu}{\nu+1}}$.

Observe that as $\nu\to\infty$ this upper bound approaches 
asymptotically the optimal rate $1/\log(p)$ achieved under the 
Berman condition (see Section \ref{sec:optimal}). Our results yield, 
however,  an upper bound on the rate of concentration in 
\eqref{def: URS lim} for the case $0<\nu<1,$ where the Berman
condition does not hold.
\end{example}

\subsection{Functions of Gaussian arrays}\label{sec: functions Gaussian}

The main motivation behind the work in this section is to determine when the concentration of maxima property is preserved under transformations. Specifically, consider the triangular array \begin{equation}\label{def: Eta}
\mathcal{H}=\left\{\eta_p(j) = f(\epsilon_p(j)),\ j\in[p],\ p\in\mathbb{N}\right\},\end{equation}
where $\mathcal{E}=\left\{\epsilon_p(j),\ j\in [p], \ p\in\mathbb{N}\right\}$ is a Gaussian triangular array with standard Normal marginals.

Given that \eqref{def: URS with delta_p} holds, our goal is to find bounds on a sequence $d_p\downarrow 0$, such that \begin{equation}\label{eq: eta d_p rate}
{\mathbb P}\left(\left|\frac{\max_{j\in[p]}\eta_p(j)}{v_p}-1\right|>d_p\right)\rightarrow 0,\quad\mbox{as } p\to\infty, 
\end{equation}
where $v_p=f(u_p)$  and $u_p$ is as in \eqref{def: up}.
We first address the case of monotone non-decreasing transformations.

\begin{proposition}\label{prop: d_p I}
Asssume that $f$ is a non-decreasing differentiable and eventually strictly increasing function, with $\lim_{x\to\infty}f(x)\neq 0$ and the derivative $f'(x)$ is either eventually increasing or eventually decreasing as $x\to\infty$. If \eqref{def: URS with delta_p} holds with some $\delta_p>0$, then  \eqref{eq: eta d_p rate} holds provided that \begin{equation}\label{def: d_p star}
d_p\geq d_p^{\star}:=\frac{u_p\delta_p\max\left\{|f'(u_p(1-\delta_p)|,|f'(u_p(1+\delta_p)|\right\}}{|f(u_p)|}.    
\end{equation}
\end{proposition}

\begin{proof}
Since $u_p\uparrow\infty$, by the monotonicity of $f$ and the fact that it is eventually strictly increasing, one can show that $f(u_p) = v_p = F_{\eta}^{\leftarrow}(1-1/p),$ for $p$ large enough. 
We start by noticing that \begin{align}\label{eq: max_eta = f(max_epsilon)}
\left|\frac{\max_{j\in [p]}\eta_p(j)}{v_p}-1\right| & = \left|\frac{\max_{j\in [p]}f(\epsilon_p(j))-f(u_p)}{f(u_p)}\right| = \left|\frac{f\left(\max_{j\in [p]}\epsilon_p(j)\right)-f(u_p)}{f(u_p)}\right|,
\end{align}
where the second equality follows by the monotonicity of $f$. 

Now recall that $f$ is differentiable. By the Mean Value Theorem, there exists a possibly random $\theta_p$ between $u_p$ and $\max_{j\in [p]}\epsilon_p(j)$, such that 
\begin{equation}\label{eq: MVT}
\left|\frac{f\left(\max_{j\in [p]}\epsilon_p(j)\right)-f(u_p)}{f(u_p)}\right|=\left|\frac{1}{f(u_p)}f'(\theta_p)\left(\max_{j\in [p]}\epsilon_p(j)-u_p\right)\right|. 
\end{equation}

Combining \eqref{eq: max_eta = f(max_epsilon)} and \eqref{eq: MVT}, we obtain
 \begin{align*}
 {\mathbb P}\left(\left|\frac{\max_{j\in[p]}\eta_p(j)}{v_p}-1\right|>d_p\right) & ={\mathbb P}\left(\left|\frac{u_p f'(\theta_p)}{f(u_p)}\right|\cdot\left|\frac{\max_{j\in[p]}\epsilon_p(j)}{u_p}-1\right|>d_p\right)\\ & =
 {\mathbb P}\left(\left|\frac{\max_{j\in[p]}\epsilon_p(j)}{u_p}-1\right|>\frac{d_p |f(u_p)|}{u_p|f'(\theta_p)|}\right),
 \end{align*}
 where the second equality follows from the fact that $f'(\theta_p)\neq0$ over the event of interest, since $d_p>0$.
This shows that for any non-negative sequence $\delta_p$ vanishing to 0, such that \eqref{def: URS with delta_p} holds,
we have that \begin{equation}\label{conv: concentration tilde d_p}
{\mathbb P}\left(\left|\frac{\max_{j\in [p]}\eta_p(j)}{v_p}-1\right|>\tilde d_p\right)\to 0, \quad\mbox{as } p\to\infty,
\end{equation}
where \begin{equation}\label{def: tilde d_p}
\tilde d_p:=\frac{u_p\delta_p|f'(\theta_p)|}{|f(u_p)|.}    
\end{equation}

Now, we know by \eqref{def: URS with delta_p} that $$\left|\theta_p-u_p\right|\leq\left|\max_{j\in [p]}\epsilon_p(j)-u_p\right|\leq u_p\delta_p$$ with probability going to 1, as $p\to\infty.$ This implies that  $${\mathbb P}\left(u_p(1-\delta_p)\leq\theta_p\leq u_p(1+\delta_p)\right)\to1,\quad\mbox{as } p\to\infty,$$
In turn, by the eventual monotonicity of $f'$, the last convergence implies that 
\begin{equation*}
{\mathbb P}\left(\left|f'(\theta_p)\right|\leq\max\left\{\left|f'(u_p(1-\delta_p))\right|,\left|f'(u_p(1+\delta_p))\right|\right\}\right)\to 1,\quad\mbox{as } p\to\infty,  
\end{equation*}
and equivalently 
\begin{equation}\label{conv: d_p tilde d_p star} 
{\mathbb P}\left(\tilde d_p\leq d_p^{\star}\right)\to1,\quad\mbox{as } p\to\infty.
\end{equation}
By \eqref{def: tilde d_p} and \eqref{conv: d_p tilde d_p star}
we conclude that \eqref{conv: concentration tilde d_p} holds with $\tilde d_p$ substituted by $d_p^{\star}$. This shows that $d_p^{\star}$ is an upper bound of the optimal rate of concentration, i.e., \eqref{def: d_p star} implies \eqref{eq: eta d_p rate}. \qed
\end{proof}

A typical and very important case where Proposition \ref{prop: d_p I} applies is when the array $\mathcal{E}$ undergoes an exponential transformation, illustrated in the following example. 

\begin{example}\label{ex: lognormal}
Let $\mathcal{E}$ be as in Proposition \ref{prop: d_p I} and consider \begin{equation}\label{ex: lognormal: eta_array}
\mathcal{H}_{\mathcal{E}}=\left\{\eta_p(j):=e^{\epsilon_p(j)},\ j\in[p],\ p\in\mathbb{N}\right\},    
\end{equation}
which is a triangular array with lognormal marginal distributions. This is sometimes referred to as the multivariate lognormal model \citep{halliwell2015lognormal}. Let $\delta_p$ be such that \eqref{def: URS with delta_p} holds. Then, an immediate application of Proposition \ref{prop: d_p I} shows that as long as $u_p\delta_p\to0,$ an upper bound on the rate of convergence in \eqref{eq: eta d_p rate} is $$d_p^{\star}=u_p\delta_pe^{u_p\delta_p}\sim u_p\delta_p\sim\delta_p\sqrt{2\log(p)}.$$ That is, lognormal arrays can have relatively stable maxima, provided that the underlying maxima of the Gaussian array concentrate at a rate $\delta_p=o\left(1/\sqrt{\log(p)}\right).$   
\end{example}

Popular models like the ones with $\chi_1^2$ marginals can be obtained from Proposition \ref{prop: d_p I} with the monotone transformation $f(x):=F^{-1}\left(\Phi(x)\right)$, where $F$ is the cdf of the desired distribution. The classic multivariate $\chi_1^2$- models, however, are obtained by squaring the elements of the Gaussian array, i.e., via the non-monotone transformation $f(x)=x^2.$ Such models are addressed in the next result. 

\begin{corollary}\label{cor: d_p II}
Let all the assumptions of Proposition \ref{prop: d_p I} hold and let $d_p^{\star}$ be defined as before. Assume now that $f$ is an even ($f(x)=f(-x)$) differentiable and eventually strictly increasing function, with $\lim_{x\to\infty}f(x)\neq0$. Assume also that $f$ is monotone non-decreasing on $(0,\infty)$. Then, the conclusion \eqref{def: d_p star} still holds.
\end{corollary}
\begin{proof}
We start by observing that   
\begin{align}\label{eq: max_eta = f(max_epsilon) II}
 & {\mathbb P}\left(\left|\frac{\max_{j\in [p]}\eta_p(j)}{f(u_p)} -1\right|>d_p\right) = {\mathbb P}\left(\left|\frac{\max_{j\in [p]}f(\epsilon_p(j))-f(u_p)}{f(u_p)}\right|>d_p\right) \nonumber \\ & \leq {\mathbb P}\left(\left|\frac{f(\min_{j\in [p]}\epsilon_p(j))-f(u_p)}{f(u_p)}\right|>d_p\right)+{\mathbb P}\left(\left|\frac{f(\max_{j\in [p]}\epsilon_p(j))-f(u_p)}{f(u_p)}\right|>d_p\right),
\end{align}
because the symmetry and monotonicity of $f$ on $(0,\infty)$ imply that $\max_{j\in[p]}f(\epsilon_p(j))$ equals either $f\left(\max_{j\in[p]}\epsilon_p(j)\right)$ or $f\left(\min_{j\in[p]}\epsilon_p(j)\right).$

By Proposition \ref{prop: d_p I} we can readily obtain that for $d_p\geq d_p^{\star}$ the second term of \eqref{eq: max_eta = f(max_epsilon) II} converges to 0. Now, we handle the first term of \eqref{eq: max_eta = f(max_epsilon) II}. By the symmetry of $f$ we have that $$f(\min_{j\in[p]}\epsilon_p(j))=f(-\min_{j\in [p]}\epsilon_p(j))=f(\max_{j\in [p]}(-\epsilon_p(j)).$$ 
Notice that by verifying the equality of the covariance structures, we have $$\left\{-\epsilon_p(j),\ j\in[p]\right\}\stackrel{d}{=}\left\{\epsilon_p(j),\ j\in[p]\right\}.$$
Hence $\max_{j\in[p]}(-\epsilon_p(j))\stackrel{d}{=}\max_{j\in[p]}\epsilon_p(j),$ and again by Proposition \ref{prop: d_p I} we get that for $d_p\geq d_p^{\star}$ the first term of \eqref{eq: max_eta = f(max_epsilon) II} also converges to 0. This completes the proof. \qed
\end{proof}

Using Corollary \ref{cor: d_p II} we can now treat the  multivariate $\chi^2$ model introduced in \cite{doi:10.1080/03610929708832000}.
\begin{example}\label{ex: chi-square}
Let $\mathcal{E}$ be as in Proposition \ref{prop: d_p I} and consider $$\mathcal{H}_{\mathcal{E}}=\left\{\eta_p(j):=\epsilon_p^2(j),\ j\in[p],\ p\in\mathbb{N}\right\},$$ a triangular array with $\chi_1^2$ marginal distributions. Let $\delta_p$ be as in \eqref{def: URS with delta_p}.  Then, a simple application of Corollary \ref{cor: d_p II} implies \eqref{eq: eta d_p rate}, provided   $$d_p\geq d_p^{\star}=2\delta_p(1+\delta_p)\sim2\delta_p.$$ In contrast to Example \ref{ex: lognormal}, taking squares does not lead to a slower rate of convergence. Indeed, in Example \ref{ex: lognormal} our estimate of the rate is slowed down by a factor of $\sqrt{\log(p)}$, while in the $\chi^2$ case it remains $\delta_p$.
\end{example}

We shall now see that the rate of convergence is not slowed down by any power transformation $x\mapsto x^{\lambda},$ for any $\lambda>0.$

\begin{example}\label{ex: powers}\textit{Power-Law Transformations.}\\
Let once again $\mathcal{E}$ be as in Proposition \ref{prop: d_p I} and consider the power transformations $f(x) = x^{\lambda},\ \lambda>0.$ In the cases where $\lambda\not\in\mathbb{N},$ we use the functions $f_1^{\lambda}(x)= |x|^{\lambda}$ or $f_2^{\lambda}(x) = x^{<\lambda>}=\rm{sign}(x)\cdot |x|^{\lambda}.$ Note that differentiability at 0 is not needed in any of the proofs, so using $f_1^{\lambda}$ does not violate any of the assumptions. Let also $\delta_p$ be as in \eqref{def: URS with delta_p}, i.e., a rate sequence for the convergence in \eqref{def: URS lim}.  Then, a suitable application of Proposition \ref{prop: d_p I} or Corollary \ref{cor: d_p II}, shows that an upper bound on the rate of convergence in \eqref{eq: eta d_p rate} is $$d_p^{\star}=\lambda\delta_p(1+\delta_p)^{\lambda-1}\sim \lambda\delta_p\quad\quad \text{or} \quad\quad d_p^{\star}=\lambda\delta_p(1-\delta_p)^{\lambda-1}\sim \lambda\delta_p.$$

In view of Examples \ref{ex: iid Gaussian}, \ref{ex: cov k^-gamma} and \ref{ex: cov 1/(logk)^nu}, we now show how the rate $d_p^{\star}\sim\lambda\delta_p$ is affected under different correlation structures of the underlying Gaussian array $\mathcal{E}.$ Recall that in the iid case of Example \ref{ex: iid Gaussian} we have that the optimal rate is $\delta_p \gg \delta_p^{{\rm opt}} = 1/\log(p).$ This implies that an upper bound on the rate of concentration is $$d_p^{\star}\sim\lambda\delta_p\gg \frac{\lambda}{\log(p)}.$$   
Moreover, for the power-law covariance decay covariance structure (Example \ref{ex: cov k^-gamma}), we observe that compared to the iid case, the rate of concentration $\delta_p$ is scaled by a factor of $\log(\log(p))$. Namely, for the power-law transformations we get that the upper bound is 
$$
d_p^{\star}\sim\lambda\delta_p\sim \frac{\lambda \log(\log(p))}{\log(p)}.
$$

Finally, we examine the logarithmic covariance decay (Example \ref{ex: cov 1/(logk)^nu}). Remember that in this case the rate we have for $\mathcal{E}$ is $\delta_p =\left(\log(p)\right)^{-\frac{\nu}{\nu+1}}.$ This implies that the upper bound of the rate of concentration for the power-law transformations is  
$$d_p^{\star}\sim\lambda\delta_p\sim \frac{\lambda }{\left(\log(p)\right)^{\frac{\nu}{\nu+1}}}. $$
Observe that in this case, $d_p^{\star}$ is a valid upper bound aside from the value of $\nu.$ We will see in the following Example \ref{ex: exponential powers}, that the same is not true for the exponential power-law transformations. 
\end{example}

In the last example of this section, we explore exponential power transformations and how they affect our bounds on the rate of convergence.  

\begin{example}\label{ex: exponential powers}\textit{Exponential Power-Law Transformations.}\\
Let $\mathcal{E}$ be as in Proposition \ref{prop: d_p I} and consider the exponential power transformations $f(x) = e^{x^{\lambda}},\ \lambda>0,\ \lambda\neq1.$ (Note that $\lambda=1$ is the lognormal case which we have alredy seen in Example \ref{ex: lognormal}). In the cases where $\lambda\not\in\mathbb{N},$ we use the functions $f_1^{\lambda}(x)= e^{|x|^{\lambda}}$ or $f_2^{\lambda}(x) = e^{x^{<\lambda>}}=e^{\rm{sign}(x)\cdot |x|^{\lambda}}.$ Similarly to Example \ref{ex: powers}, differentiability at 0 is not needed in any of the proofs, so using $f_1^{\lambda}$ does not violate any of the assumptions. Let also $\delta_p$ be as in \eqref{def: URS with delta_p}.  Then, suitable applications of Proposition \ref{prop: d_p I} or Corollary \ref{cor: d_p II} show that as long as $u_p^{\lambda}\delta_p\to 0,$ an upper bound on the rate of convergence in \eqref{eq: eta d_p rate} is $$d_p^{\star}=\lambda u_p^{\lambda}\delta_p(1+\delta_p)^{\lambda-1}e^{u_p^{\lambda}\left[(1+\delta_p)^{\lambda}-1\right]},\quad \mbox{if }\lambda\geq 1$$
and $$d_p^{\star}=\lambda u_p^{\lambda}\delta_p(1-\delta_p)^{\lambda-1}e^{u_p^{\lambda}\left[(1-\delta_p)^{\lambda}-1\right]},\quad\mbox{if } 0<\lambda<1.$$
In both cases we have $d_p^{\star}\sim\lambda\delta_p(2\log(p))^{\lambda/2},$ as $p\to\infty$.
As a generalization of the lognormal case ($\lambda=1$), we see that the iid\ rate $\delta_p$ is scaled by a factor of $\left(\sqrt{\log(p)}\right)^{\lambda}$. This means that this kind of arrays would still have relatively stable maxima, provided that the underlying maxima of the Gaussian array concentrate at a rate $\delta_p=o\left(1/(\log(p))^{\lambda/2}\right).$

At this point, we examine how the rate $d_p^{\star}\sim \lambda\delta_p(2\log(p))^{\lambda/2}$ adjusts under the varying covariance structures of $\mathcal{E}$ in Examples \ref{ex: iid Gaussian}, \ref{ex: cov k^-gamma} and \ref{ex: cov 1/(logk)^nu}. In an analogous manner to Example \ref{ex: powers}, we get that for the iid case, an upper bound on the rate of concentration is $$d_p^{\star}\sim\lambda\delta_p u_p^{\lambda}\gg 2^{\frac{\lambda}{2}}\lambda \left(\log(p)\right)^{\frac{\lambda}{2}-1},$$
while for the power-law covariance decay covariance structure we obtain 
$$
d_p^{\star}\sim\lambda\delta_p u_p^{\lambda}\sim 2^{\frac{\lambda}{2}}\lambda \left(\log(p)\right)^{\frac{\lambda}{2}-1}\log(\log(p)).$$

In the previous two instances we notice that the covariance structure does not impose any restrictions on the values of $\lambda$, in order to guarantee concentration of maxima for the transformed triangular array.
This is not the case for the logarithmic covariance decay, since the upper bound becomes $$d_p^{\star}\sim\lambda\delta_p u_p^{\lambda}\sim 2^{\frac{\lambda}{2}}\lambda \left(\log(p)\right)^{\frac{\lambda}{2}-\frac{\nu}{\nu+1}}.$$

The aforementioned $d_p^{\star}$ is a a sensible upper bound for the rate of concentration in this case, only if $d_p^{\star}\to0,$ as $p\to\infty$. This is so, when $\nu>\frac{\lambda}{2+\lambda}$. Thus, our results imply that in the lognormal case ($\lambda=1)$, $\nu>\frac{1}{3}$ guarantees that the transformed array is relatively stable.  
\end{example}

\begin{remark}
In Conjecture \ref{con: optimal rate}, we posit that the fastest rate of convergence for a UDD Gaussian array is bounded above by $1/\log(p)$. Nevertheless, from Example \ref{ex: iid Gaussian} for the iid case, our bound in \eqref{ineq: lb on the rate} is again $1/\log(p).$ Since $u_p\sim\sqrt{2\log(p)},$ we see that we can get an upper bound on the rate of $f(x)=e^{x^{\lambda}}$ only for $0<\lambda<2.$ The range $\lambda\in(0,2)$ is also natural, because one can show that the transformation $f(x)=e^{x^{\lambda}},$ for $\lambda\geq 2$, leads to heavy power-law distributed variables $\eta_p(j).$ Heavy-tailed random variables no longer have relatively stable maxima, which makes the question about the rate of concentration of maxima meaningless. 
\end{remark}

We will end this section with a corollary, readily obtained by the discussion in the end of Example \ref{ex: exponential powers}. 
\begin{corollary}\label{cor: log-normal}
Suppose that $\mathcal{H}:=\left\{\eta_p(j),\ j\in[p],\ p\in\mathbb{N}\right\}$ is a multivariate log-normal array as in \eqref{ex: lognormal: eta_array}. Suppose that \begin{equation}\label{cor: log-normal: cov eta inequality}
\left|{\rm Cov}\left(\eta_p(j),\eta_p(k)\right)\right|\leq c\cdot\frac{1}{\left(\log(|\pi_p(j)-\pi_p(k)|)\right)^{\nu}},    
\end{equation}
for some $\nu>1/3,$ permutations $\pi_p$ of $\left\{1,\hdots,p\right\}$ and a constant $c$ independent of $p$. Then the array $\mathcal{H}$ is URS. 
\end{corollary}
\begin{proof}
Let $\mathcal{E}=\left\{\epsilon_p(j),\ j\in[p],\ p\in\mathbb{N
}\right\}$ be the underlying Gaussian array. Then, we have that $\eta_p(j)=e^{\epsilon_p(j)}$ for every $j\in[p]$. Thus, \begin{align}\label{cor: log-normal: Cov eta as Cov epsilon}
{\rm Cov}(\eta_p(j),\eta_p(k))&={\rm Cov\left(e^{\epsilon_p(j)},e^{\epsilon_p(k)}\right)}\nonumber\\
&=\E\left(e^{\epsilon_p(j)+\epsilon_p(k)}\right)-\E\left(e^{\epsilon_p(j)}\right)\E\left(e^{\epsilon_p(k)}\right).
\end{align}
Recall that the moment generating function for a Normal random variable  $X\sim N(\mu,\sigma^2)$ is $M(t)=\E\left(e^{tX}\right)=e^{\mu t+\sigma^2t^2/2}$. Since $\epsilon_p(i)$ follow the standard Normal distribution, we have $\epsilon_p(j)+\epsilon_p(k)\sim N(0,2+2{\rm Cov}(\epsilon_p(j),\epsilon_p(k)))$, and hence \eqref{cor: log-normal: Cov eta as Cov epsilon} becomes \begin{equation}\label{cor: log-normal: Cov eta as Cov epsilon final} 
{\rm Cov}(\eta_p(j),\eta_p(k))= e\cdot\left(e^{{\rm Cov}(\epsilon_p(j),\epsilon_p(k))}-1\right).    
\end{equation}
In turn, \eqref{cor: log-normal: Cov eta as Cov epsilon final}
along with \eqref{cor: log-normal: cov eta inequality} implies that
\begin{equation}\label{cor: log-normal: penultimate step}
\left|e\cdot\left(e^{{\rm Cov}(\epsilon_p(j),\epsilon_p(k))}-1\right)\right| \leq \frac{c}{e}\cdot \frac{1}{\left(\log(|\pi_p(j)-\pi_p(k)|)\right)^{\nu}}.    
\end{equation}
Using the inequality $|x|\leq e|e^x-1|,\ x\in[-1,1]$ in \eqref{cor: log-normal: penultimate step}, since $\left|{\rm Cov}(\epsilon_p(j),\epsilon_p(k))\right|\leq 1$,
we finally obtain that 
$$\left|{\rm Cov}(\epsilon_p(j),\epsilon_p(k))\right|\leq \frac{c}{e}\cdot \frac{1}{\left(\log(|\pi_p(j)-\pi_p(k)|)\right)^{\nu}}.$$
The last relation implies that $\mathcal{E}$ has a logarithmic covariance decay covariance structure (see Example \ref{ex: cov 1/(logk)^nu}). Combined with the discussion in the end of Example \ref{ex: exponential powers}, the proof is complete. \qed
\end{proof}

\section{Technical proofs} \label{sec:proofs}
In this section we present the proof of the capstone Theorem \ref{thm: capstone}. Recall that we desire to find an upper bound on the rate of positive vanishing sequences $\delta_p$, such that 
$${\mathbb P}\left(\left|\frac{\max_{i\in[p]}\epsilon_p(i)}{u_p}-1\right|>\delta_p\right)\to 0, \quad\mbox{as } p\to\infty.$$ 
To this end, let \begin{equation}\label{def: xip}
\xi_p:=\frac{1}{u_p}\max_{i\in[p]}\epsilon_p(i),    
\end{equation}
where $\mathcal{E}=\left\{\epsilon_p(i),\, i\in[p]\right\}$ is a URS Gaussian array with standard Normal marginals.
Observe that \begin{align}
{\mathbb P}(|\xi_p-1|>\delta_p) & = {\mathbb P}(\xi_p>1+\delta_p) + {\mathbb P}(\xi_p<1-\delta_p)\nonumber\\\label{def: I and II}
                        & =: {\rm I}(\delta_p) + {\rm II}(\delta_p). 
\end{align} 
Thus, to obtain the desired rate we need to recover a bound on the rate of ${\rm I}(\delta_p)$ and ${\rm II}(\delta_p)$. Note that in our endeavor to secure upper bound on the term ${\rm II}(\delta_p)$ we will use the expectation of $\xi_p.$ The integrability of $\xi_p$ is ensured by Appendix A.2 of \cite{chatterjee2014superconcentration}, or \cite{pickands1968moment} in conjunction with \eqref{eq: URS S_p}.\\
 
\textbf{Term ${\rm I}(\delta_p)$.} In the following proposition, we find an upper bound on the rate of $\delta_p$ in I$(\delta_p)$ of \eqref{def: I and II}. Interestingly, the following result does not involve the dependence structure of the array $\mathcal{E}$.

\begin{proposition}\label{prop: UpperBound}
Let $\mathcal{E}=\{\epsilon_p(i),\ i\in[p]\}$ be an arbitrary Gaussian triangular array, where the marginal distributions are standard Normal and let $\xi_p$ be defined as in \eqref{def: xip}. If $\delta_p\to0$ is a positive sequence such that  \begin{equation}\label{assumption delta_p}
\delta_p\gg\frac{1}{\log(p)}   
\end{equation} then, regardless of the dependence structure of $\mathcal{E},$ we have  
\begin{equation}\label{e:prop:UpperBound}
 \lim_{p\to\infty}\left(\delta_p^{-1}\E (\xi_p-1)_+\right) =0,
\end{equation}
and consequently
${\mathbb P} ( \xi_p > 1 + \delta_p ) \to 0,$ as $p\to\infty$.
\end{proposition}

\noindent 
We need the following simple bound for the Mill's ratio.
\begin{lemma}\label{l:Mills_ratio} For all $u>0$, we have
$$
1- \frac{1}{1\vee u^2} \le \frac{\overline{\Phi}(u)}{\phi(u)/u}\le 1,
$$
where $\phi(u) = e^{-u^2/2}/\sqrt{2\pi}$ and $\overline\Phi(u) = \int_u^\infty\phi(x)dx.$
\end{lemma}
\begin{proof}  We have
\begin{align*}
\frac{\overline{\Phi}(u)}{\phi(u)/u} &= \frac{u}{\phi(u)} \int_u^{\infty} \phi(x) dx = u \int_u^{\infty} e^{-\frac{x^2 - u^2}{2}} dx \\
&= u \int_0^{\infty} e^{-\frac{(z+u)^2 - u^2}{2}} dz = \int_0^\infty e^{-\frac{z^2}{2}} ue^{-uz} dz = \E [ e^{-E^2/(2u^2)}],
\end{align*}
where $E$ is an exponentially distributed random variable with unit mean, and we used the change 
of variables $z:= x-u$. Observing that $1 - x \le e^{-x} \le 1$, for all $x\ge 0$, we get 
$$
1- \frac{E^2}{2u^2} \le e^{-E^2/(2u^2)} \le 1.
$$
The result follows upon taking expectation and recalling that $\E[E^2] = 2$.
\end{proof}

\begin{proof}[Proposition \ref{prop: UpperBound}]
Note first that \eqref{e:prop:UpperBound} implies $\P(\xi_p>1+\delta_p)\to 0$.  Indeed, this follows from the Markov inequality:
$$
\P(\xi_p -1 > \delta_p) = \P( (\xi_p-1)_+ > \delta_p) \le \delta_p^{-1}\E (\xi_p-1)_+.
$$

Now, we focus on proving \eqref{e:prop:UpperBound}.
We can write \begin{align}\label{def: Jdp}
\frac{1}{\delta_p}\E(\xi_p-1)_+& 
=\frac{1}{\delta_p}\int_0^{\infty}{\mathbb P}(\xi_p-1>z)dz\nonumber\\ &  = \int_0^{\infty}{\mathbb P}(\xi_p>1+\delta_px)dx=:\text{J}(\delta_p),
\end{align}
where in the last integral we used the change of variables $z=\delta_p x$.

Recalling that $\xi_p = u_p^{-1}\max_{i\in [p]} \epsilon_p(i)$, 
by the union bound, for the last integrand we have that 
\begin{equation}\label{e:P-xi-bound}
{\mathbb P}(\xi_p>1+\delta_px)\leq  p\overline\Phi(u_p(1+\delta_px)) =\frac{\overline\Phi(u_p(1+\delta_px))}{\overline\Phi(u_p)}.
\end{equation}

By Lemma \ref{l:Mills_ratio}, we further obtain that 
\begin{align}\label{e:Bp-bound}
 \frac{\overline\Phi(u_p(1+\delta_px))}{\overline\Phi(u_p)}& \leq  \frac{1}{1-1/(1\vee u_p^2)}\cdot \frac{\phi(u_p(1+\delta_px))}{(1+\delta_px)\phi(u_p)}\nonumber \\
    & \leq  \frac{1}{1-1/(1\vee u_p^2)}
    \exp\Big\{ - \frac{u_p^2}{2}\Big( (1+\delta_px)^2 - 1\Big) \Big\}\nonumber\\
    & \leq B_p \exp\{ - u_p^2 \delta_p x\},
\end{align}
where $B_p:=(1-1/(1\vee u_p^2))^{-1}\to 1$, as $p\to\infty$, 
is a constant independent of $x\geq 0$ and in the last inequality 
we also used the simple bound $(1+\delta_px)^2 - 1\geq 2 \delta_px$.

Condition \eqref{assumption delta_p} means that there is a sequence $\gamma(p)$ diverging to infinity slower than $\log(p)$ such that 
$$
    \label{e:delta-gamma}
    \delta(p) = \frac{\gamma(p)}{\log(p)}.
$$
Thus, by Relation \eqref{e:Bp-bound} and the facts that $u_p^2\sim 2 \log(p)$ and $B_p\sim 1$, as $p\to\infty$, we obtain
\begin{align*}
    \frac{\overline\Phi(u_p(1+\delta_px))}{\overline\Phi(u_p)}& \leq 2 \cdot e^{-2\gamma(p)x},
\end{align*}
for all sufficiently large $p$. Since $\gamma(p)\to \infty$, Relation
\eqref{e:P-xi-bound} and the Dominated Convergence Theorem applied to
\eqref{def: Jdp}, implies
\begin{align*}
\lim_{p\to\infty} \text{J}(\delta_p) & \leq \lim_{p\to\infty}
\int_0^{\infty} 2e^{-2\gamma(p) x}dx = 0.
\end{align*}
This completes the proof of \eqref{e:Bp-bound}. \qed
\end{proof}

\textbf{Term ${\rm II}(\delta_p)$.} Handling term II of \eqref{def: I and II} is more involved and this is where the dependence structure of the array plays a role. We start by presenting a more careful reformulation of Lemma B.1 in \cite{gao2018fundamental}. 

\begin{lemma}\label{lem: B.1}
Let $(X_i)_{i=1}^p$ be $p$ iid random variables with distribution $F$ and density $f$, such that $$\E(X_i)_-\equiv\E(\max\{-X_i,0\})<\infty.$$
Denote the maximum of the $X_i$'s as $M_p:=\max_{i=1,\hdots,p}X_i.$ Suppose that $f$ is eventually decreasing, i.e., there exists a $C_0$ such that $f(x_1)\geq f(x_2)$ whenever $C_0\leq x_1\leq x_2,$ then 
$$\frac{\E M_p}{u_{p+1}}\geq (1-F^p(C_0))+\frac{\E [X_1|X_1<C_0]}{u_{p+1}}F^p(C_0),$$
where $u_{p+1}=F^{\leftarrow}(1-1/(p+1)).$
\end{lemma}
\begin{proof}
For the proof, refer to the proof of Lemma B.1 in \cite{gao2018fundamental}.\qed
\end{proof}

Recall that a Gaussian triangular array $\mathcal{E}=\left\{\epsilon_p(j)\right\}_{j=1}^p$ with standard Normal marginals is said to be UDD if for every $\tau>0$, \begin{equation}\label{eq: N_epsilon}
N_{\mathcal{E}}(\tau):=\sup_{p\in\mathbb{N}}\max_{i=1,\hdots,p}\left|\left\{\kappa\in[p]: \ \rm{Cov}(\epsilon_p(i),\epsilon_p(\kappa))>\tau\right\}\right|<\infty.
\end{equation}
That is, for every $p$ and $i\in[p]$, there are at most $N_{\mathcal{E}}(\tau)$ indices $\kappa$, such that the covariance between $\epsilon_p(i)$ and $\epsilon_p(\kappa)$ exceeds $\tau$.

The function $N_{\mathcal{E}}(\tau)$ encodes certain aspects of the dependence structure of the array $\mathcal{E}.$ It will play a key role in the derivation of the upper bound on the rate of concentration of maxima. The next result is an extension of Proposition A.1 in \cite{gao2018fundamental} tailored to our needs. For the benefit of the reader, we reproduce the key argument involving a packing construction and the Sudakov-Fernique bounds, which may be of independent interest.
\begin{proposition}\label{prop: expectationlowerbound}
For every UDD Gaussian array $\mathcal{E}$, and any subset $S_p\subseteq\left\{1,\hdots,p\right\}$ with $ q=|S_p|,$ and $\tau\in(0,1)$, we have that \begin{align}
    \mathbb{E}\left[\frac{\underset{j\in S_p}{\max \epsilon_p(j)}}{u_q}\right]& \geq \frac{u_{q/N_{\mathcal{E}}(\tau)+1}}{u_q}\sqrt{1-\tau}\left(1-\frac{1}{2^{q/N_{\epsilon}(\tau)}}-\frac{\sqrt{2/\pi}}{u_{q/N_{\epsilon}(\tau)+1}}\cdot\frac{1}{2^{q/N_{\epsilon}(\tau)}}\right)\label{ineq: basicinequality}\\
    &:=1-R_q\label{def: Rp},
\end{align}
where $N_{\mathcal{E}}(\tau)$ is given in \eqref{eq: N_epsilon}.
\end{proposition}
\begin{remark}
Note that without loss of generality we can assume $S_p=\{1,\hdots,p\}.$ We prove a slightly more general result, but the only application in this paper will be for $q=p.$
\end{remark}
\begin{proof}
Define the canonical (pseudo) metric on $S_p$, 
$$d(i,j)=\sqrt{\E (\epsilon(i)-\epsilon(j))^2}.$$ 
This metric takes values between 0 and 2, since $\epsilon_p(i),\ i=1,\hdots,p,$ have zero means and unit variances. Fix $\tau\in(0,1)$, take $\gamma=\sqrt{2(1-\tau)}$ and let $\Gamma$ be a $\gamma$-packing of $S_p$. That is, let $\Gamma$ be a subset of $S_p$, such that for any $i,j\in \Gamma,\ i\neq j,$ we have $d(i,j)>\gamma,$ i.e.,
$$d(i,j)=\sqrt{2\left(1-\Sigma_p(i,j)\right)}\geq\gamma=\sqrt{2(1-\tau)},$$
or equivalently, $\Sigma_p(i,j)\leq \tau.$ We claim that we can find a $\gamma$-packing $\Gamma$ whose number of elements is at least \begin{equation}\label{def: Gamma lower bound}
|\Gamma|\geq \frac{q}{N_{\mathcal{E}}(\tau)}.    
\end{equation}
Indeed, $\Gamma$ can be constructed iteratively as follows:
\begin{enumerate}[label= \textbf{Step \arabic*:},leftmargin=*]
    \item Set $S_p^{(0)}:=S_p$ and $\Gamma:=\{j_1\},$ where $j_1\in S_p^{(0)}$ is an arbitrary element. Set $k:=1$.
    \item Set $S_p^{(k)}:=S_p^{(k-1)}\setminus B_{\gamma}(j_k)$, where $$B_{\gamma}(j_k):=\{i\in S_p: d(i,j_k)<\gamma\}.$$
    \item If $S_p^{(k)}\neq\emptyset$, pick an arbitrary $j_{k+1}\in S_p^{(k)},$ set $\Gamma:=\Gamma\cup\{j_{k+1}\},$ and $k:=k+1,$ go to Step 2; otherwise stop.
\end{enumerate}
By the definition of UDD, there are at most $N_{\mathcal{E}}(\tau)$ coordinates whose covariance with $\epsilon_p(j)$ exceed $\tau.$ Therefore, at each iteration, $|B_{\gamma}(j_k)|\leq N_{\mathcal{E}}(\tau),$ and hence $$|S_p^{(k)}|\geq |S_p^{(k-1)}|-|B_{\gamma}(j_k)|\geq q-kN_{\mathcal{E}}(\tau).$$
The construction can continue for at least $q/N_{\mathcal{E}}(\tau)$ iterations, which implies \eqref{def: Gamma lower bound}.

Now, we define on this $\gamma$-packing $\Gamma$ an independent Gaussian process $\{\eta(j)\}_{j\in\Gamma}$, $$\eta(j)=\frac{\gamma}{\sqrt{2}}Z(j), \ j\in\Gamma,$$
where the $Z(j)$'s are iid standard Normal random variables. The increments of the new process are smaller than that of the original in the following sense, 
$$\E(\eta(i)-\eta(j))^2=\gamma^2\leq d^2(i,j)=\E(\epsilon_p(i)-\epsilon_p(j))^2,$$
for all $i\neq j,\ i,j\in\Gamma.$ Applying the Sudakov-Fernique inequality (see, e.g., Theorem 2.2.3 in \cite{adler2009random}) to $\{\eta(j)\}_{j\in\Gamma}$ and $\{\epsilon_p(j)\}_{j\in\Gamma},$ we have 
$$\mathbb{E}\left[\underset{j\in \Gamma}{\text{max}}(\eta(j))\right]\leq \mathbb{E}\left[\underset{j\in\Gamma}{\text{max}}(\epsilon_p(j))\right]\leq \mathbb{E}\left[\underset{j\in S_p }{\text{max}}(\epsilon_p(j))\right].$$

This implies 
\begin{align*}
\mathbb{E}\left[\frac{1}{u_q}\max_{j\in S_p} \epsilon_p(j)\right]\geq \mathbb{E}\left[\frac{1}{u_{|\Gamma|+1}}\max_{j\in \Gamma} \eta(j)\right] \cdot\frac{u_{|\Gamma|+1}}{u_q}\geq \frac{u_{|\Gamma|+1}}{u_q}\cdot\sqrt{1-\tau}\cdot \mathbb{E}\left[\frac{1}{u_{|\Gamma|+1}}\max_{j\in \Gamma} Z(j)\right].
\end{align*}
Now, the application of Lemma \ref{lem: B.1} to the standard Normal distribution for $C_0=0$ entails that,
\begin{align*}
    \frac{\E \left[\max_{j\in \Gamma} Z(j)\right]}{u_{|\Gamma|+1}}\geq 1-\frac{1}{2^{|\Gamma|}}-\frac{\sqrt{2/\pi}}{u_{|\Gamma|+1}}\cdot\frac{1}{2^{|\Gamma|}}.
\end{align*}
Since $|\Gamma|\geq q/N_{\mathcal{E}}(\tau)$ the desired lower bound in \eqref{ineq: basicinequality} is obtained.\qed
\end{proof}

We are now interested in the rate at which the lower bound in \eqref{ineq: basicinequality} converges to 1. Equivalently, we desire to find the rate of decay of $R_q.$ This rate is obtained in the following Lemma.

\begin{lemma}\label{lem: rate for II}
Let $R_q,\,\alpha(q)$ be defined as in \eqref{def: Rp} and \eqref{def: alpha(p)} respectively. Then \begin{equation}\label{def: upper bound of rate of R_p}
R_q\asymp \alpha(q)+\tau(q)+2^{-q^{1-\alpha(q)}},\quad\mbox{as } q\to\infty. 
\end{equation}
\end{lemma}
\begin{proof}
Note that by definition $R_q\to0$, as $q\to\infty$. This implies that $R_q\sim\log(1-R_q)$, as $q\to\infty$, so we just need the rate of \begin{align*}
& \log (1-R_q)  = \log \left(\frac{u_{q/N_{\mathcal{E}}(\tau)+1}}{u_q}\cdot\sqrt{1-\tau(q)}\cdot\left(1-\frac{1}{2^{q/N_{\epsilon}(\tau)}}-\frac{\sqrt{2/\pi}}{u_{q/N_{\epsilon}(\tau)+1}}\cdot\frac{1}{2^{q/N_{\epsilon}(\tau)}}\right)\right)\\ & \qquad = \log\left(\frac{u_{q/N_{\mathcal{E}}(\tau)+1}}{u_q}\right)+\frac{1}{2}\log(1-\tau(q))+\log\left(1-\frac{1}{2^{q/N_{\epsilon}(\tau)}}-\frac{\sqrt{2/\pi}}{u_{q/N_{\epsilon}(\tau)+1}}\cdot\frac{1}{2^{q/N_{\epsilon}(\tau)}}\right).
\end{align*}
Now, the facts that $\alpha(q)=\log(N_{\cal E}(\tau))/\log(q)$ and
$u_q\sim\sqrt{2\log(q)}$ imply that
\begin{align*}
    \frac{u_{q/N_{\mathcal{E}}(\tau)+1}}{u_q}& \sim \sqrt{\frac{2\log(1+q^{1-\alpha(q)})}{2\log(q)}} \sim \sqrt{\frac{\log(q^{1-\alpha(q)})}{\log(q)}}=\sqrt{1-\alpha(q)},
\end{align*}
where we used the relation 
\begin{equation}\label{e:q and alpha}
    q^{1-\alpha(q)} = e^{\log(q) - \log(N_{\cal E}(\tau))} = \frac{q}{N_{\cal E}(\tau)}.
\end{equation}
However, since $\alpha(q) = \log(N_{\mathcal{E}}(\tau(q))/\log(q)\to0$ and $\tau(q)\to0$, we have
    \begin{align*}
    \log(1-\alpha(q)) & =-\alpha(q)+o(\alpha(q)),\\
     \log(1-\tau(q)) & =-\tau(q)+o(\tau(q)),
     \end{align*}
and by \eqref{e:q and alpha}
\begin{align*}
    \log\left(1-\frac{1}{2^{q/N_{\epsilon}(\tau)}}-\frac{\sqrt{2/\pi}}{u_{q/N_{\epsilon}(\tau)+1}}\cdot\frac{1}{2^{q/N_{\epsilon}(\tau)}}\right) & = \log\left(1-2^{-q^{1-\alpha(q)}}-\frac{\sqrt{2/\pi}}{u_{q/N_{\epsilon}(\tau)+1}}\cdot 2^{-q^{1-\alpha(q)}}\right)\\ & = 2^{-q^{1-\alpha(q)}} + o\left(2^{-q^{1-\alpha(q)}}\right).
\end{align*}
As a result, we have 
\begin{equation}\label{lem: rate for II: intermediate rate}
R_q\asymp \alpha(q)+\tau(q)+2^{-q^{1-\alpha(q)}}+o\left(\max\left\{\alpha(q),\tau(q),2^{-q^{1-\alpha(q)}}\right\}\right),
\end{equation}
which completes the proof.
\end{proof}

\textbf{Proof of Theorem \ref{thm: capstone}.} 
We are now in position to complete the proof of Theorem \ref{thm: capstone}, which consists of a combination of the results that have already been established in Section \ref{sec:proofs}. 

\begin{proof}
Recall the definition of $\xi_p$ in \eqref{def: xip} and  that \begin{align*}
{\mathbb P}(|\xi_p-1|>\delta_p) & = {\rm I}(\delta_p) + {\rm II}(\delta_p),
\end{align*} 
where ${\rm I}(\delta_p)$ and ${\rm II}(\delta_p)$ are defined as in \eqref{def: I and II}. We shall show that both terms vanish.

Proposition \ref{prop: UpperBound}, along with \eqref{ineq: lb on the rate}, imply that I$(\delta_p)={\mathbb P}(\xi_p>1+\delta_p)\to0,$ as $p\to\infty.$
Observe that the term I$(\delta_p)={\mathbb P}(\xi_p>1+\delta_p)$ vanishes, regardless of the dependence structure of the array $\mathcal{E}$. The dependence plays a key role   in the rate of the term II$(\delta_p).$ 

We now steer our focus towards term II$(\delta_p)$. The Markov inequality yields \begin{equation*}
\text{II}(\delta_p)={\mathbb P}(\xi_p<1-\delta_p)\leq \frac{\E(\xi_p-1)_-}{\delta_p}.    
\end{equation*}
Since $\E(\xi_p-1)_-\leq \E(\xi_p-1)_+ +\bigl\lvert\E(\xi_p-1)\bigr\rvert,$ we have 
\begin{align}\label{ineq: II neg part}
\text{II}(\delta_p) & \leq \frac{1}{\delta_p}\left(\E(\xi_p-1)_++\bigl\lvert\E(\xi_p-1)\bigr\rvert\right)  \nonumber  \\
   & = \frac{1}{\delta_p}\left(\E(\xi_p-1)_+ +[\E(\xi_p-1)]_+ +[\E(\xi_p-1)]_-\right)\nonumber\\
   & \leq \frac{1}{\delta_p}\left(2\E(\xi_p-1)_++[\E(\xi_p-1)]_-\right),
\end{align}
where the last inequality follows from the fact that $[\E(\xi_p-1)]_+\leq \E(\xi_p-1)_+$.

Proposition \ref{prop: UpperBound} and \eqref{ineq: lb on the rate} imply that the term $\delta_p^{-1}\E(\xi_p-1)_+$ in \eqref{ineq: II neg part} vanishes. Moreover, Proposition \ref{prop: expectationlowerbound} entails $$[\E(\xi_p-1)]_-=\max\{0,-\E(\xi_p-1)\}\leq |R_p|.$$ Thus, the term ${\rm II}(\delta_p)$ vanishes, provided that   $R_p/\delta_p\to 0$. This follows, however, from Lemma \ref{lem: rate for II} and \eqref{ineq: lb on the rate}, since for $\alpha(p)\to 0$, we have 
$$
 \frac{1}{\log(p)}\gg 2^{-p^{1-\alpha(p)}},\quad{\rm as}\ p\to \infty
$$ and the proof is complete.\qed
\end{proof}

\begin{remark}\label{re: Tanguy} 
After we completed and submitted this paper, we became aware of the 
important work of \cite{tanguy2015some}. According to their paper, in the
stationary case, the upper bound of Theorem \ref{thm: capstone} above
partially follows from their Theorem 3. However, our work is in the
general setting of triangular arrays and does not require 
stationarity. 
The result in Theorem 5 of \cite{tanguy2015some}, could in 
principle, be used to derive bounds on rates of concentration of 
maxima for non-stationary arrays. This, however, requires verifying
two technical conditions. Our approach, based on the UDD condition
yields rates that can be explicitly related to the covariance structure 
of the array.  The in-depth comparison of the two approaches merits an
independent study beyond the scope of the present work.
\end{remark}

\begin{acknowledgements}
 We thank two anonymous referees for their very careful 
 reading of our paper. Their insightful comments led us to
 improve the presentation and the upper bound on the rate 
 of concentration in Theorem \ref{thm: capstone}. 
 
 The authors were partially supported by the NSF ATD grant DMS-1830293. The first author was also supported by the Onassis Foundation - Scholarship ID: F ZN 028-1 /2017-2018.
\end{acknowledgements}

%
%

\bibliographystyle{spbasic}      
\bibliography{references}   

\end{document}